\documentclass[a4paper]{article}

\usepackage[all]{xy}
\usepackage[leqno]{amsmath}
\usepackage{amssymb}
\usepackage[only,mapsfrom,heavycircles,lightning]{stmaryrd}
\usepackage{latexsym}
\usepackage{titlesec} 
\usepackage{paralist}
\usepackage{tocloft}
\usepackage{lmodern}
\usepackage[T1]{fontenc}
\usepackage[utf8]{inputenc}
\usepackage{calc}

\newlength{\theorempostskipamount}
\makeatletter

\newenvironment{theorem}[1][]
{\paragraph{Theorem} {\normalfont #1} \it}
{\vspace{\the\theorempostskipamount}}
\def\theorem{\@ifnextchar[{\@theoremopt}{\@theoremplain}}
\def\@theoremplain{\paragraph{Theorem} \it}
\def\@theoremopt[#1]{\paragraph{Theorem \normalfont #1}  \it}

\newenvironment{lemma}[1][]
{\paragraph{Lemma} {\normalfont #1} \it}
{\vspace{\the\theorempostskipamount}}
\def\lemma{\@ifnextchar[{\@lemmaopt}{\@lemmaplain}}
\def\@lemmaplain{\paragraph{Lemma} \it}
\def\@lemmaopt[#1]{\paragraph{Lemma \normalfont #1}  \it}

\newenvironment{proposition}[1][]
{\paragraph{Proposition} {\normalfont #1} \it}
{\vspace{\the\theorempostskipamount}}
\def\proposition{\@ifnextchar[{\@propositionopt}{\@propositionplain}}
\def\@propositionplain{\paragraph{Proposition} \it}
\def\@propositionopt[#1]{\paragraph{Proposition \normalfont #1}  \it}

\newenvironment{definition}[1][]
{\paragraph{Definition} {\normalfont #1} \it}
{\vspace{\the\theorempostskipamount}}
\def\definition{\@ifnextchar[{\@definitionopt}{\@definitionplain}}
\def\@definitionplain{\paragraph{Definition} \it}
\def\@definitionopt[#1]{\paragraph{Definition \normalfont #1}  \it}

\newenvironment{corollary}[1][]
{\paragraph{Corollary} {\normalfont #1} \it}
{\vspace{\the\theorempostskipamount}}
\def\corollary{\@ifnextchar[{\@corollaryopt}{\@corollaryplain}}
\def\@corollaryplain{\paragraph{Corollary} \it}
\def\@corollaryopt[#1]{\paragraph{Corollary \normalfont #1}  \it}

\newenvironment{question}[1][]
{\paragraph{Question} {\normalfont #1} \it}
{\vspace{\the\theorempostskipamount}}
\def\question{\@ifnextchar[{\@questionopt}{\@questionplain}}
\def\@questionplain{\paragraph{Question} \it}
\def\@questionopt[#1]{\paragraph{Question \normalfont #1}  \it}

\newenvironment{problem}[1][]
{\paragraph{Problem} {\normalfont #1} \it}
{\vspace{\the\problempostskipamount}}
\def\problem{\@ifnextchar[{\@problemopt}{\@problemplain}}
\def\@problemplain{\paragraph{Problem} \it}
\def\@problemopt[#1]{\paragraph{Problem \normalfont #1}  \it}

\newenvironment{conjecture}[1][]
{\paragraph{Conjecture} {\normalfont #1} \it}
{\vspace{\the\theorempostskipamount}}
\def\conjecture{\@ifnextchar[{\@conjectureopt}{\@conjectureplain}}
\def\@conjectureplain{\paragraph{Conjecture} \it}
\def\@conjectureopt[#1]{\paragraph{Conjecture \normalfont #1}  \it}

\newenvironment{remark}[1][]
{\paragraph{Remark} {\normalfont #1}}
{\vspace{\the\theorempostskipamount}}
\def\remark{\@ifnextchar[{\@remarkopt}{\@remarkplain}}
\def\@remarkplain{\paragraph{Remark}}
\def\@remarkopt[#1]{\paragraph{Remark \normalfont #1}}

\newenvironment{remarks}[1][]
{\paragraph{Remarks} {\normalfont #1}}
{\vspace{\the\theorempostskipamount}}
\def\remarks{\@ifnextchar[{\@remarksopt}{\@remarksplain}}
\def\@remarksplain{\paragraph{Remarks} \it}
\def\@remarksopt[#1]{\paragraph{Remarks \normalfont #1}  \it}

\newenvironment{example}[1][]
{\paragraph{Example} {\normalfont #1}}
{\vspace{\the\theorempostskipamount}}
\def\example{\@ifnextchar[{\@exampleopt}{\@exampleplain}}
\def\@exampleplain{\paragraph{Example}}
\def\@exampleopt[#1]{\paragraph{Example \normalfont #1}}

\newenvironment{examples}[1][]
{\paragraph{Examples} {\normalfont #1}}
{\vspace{\the\theorempostskipamount}}
\def\examples{\@ifnextchar[{\@examplesopt}{\@examplesplain}}
\def\@examplesplain{\paragraph{Examples} \it}
\def\@examplesopt[#1]{\paragraph{Examples \normalfont #1}  \it}

\newenvironment{exercise}[1][]
{\paragraph{Exercise} {\normalfont #1}}
{\vspace{\the\theorempostskipamount}}
\def\exercise{\@ifnextchar[{\@exerciseopt}{\@exerciseplain}}
\def\@exerciseplain{\paragraph{Exercise} \it}
\def\@exerciseopt[#1]{\paragraph{Exercise \normalfont #1}  \it}

\newenvironment{notation}[1][]
{\paragraph{Notation} {\normalfont #1}}
{\vspace{\the\theorempostskipamount}}
\def\notation{\@ifnextchar[{\@notationopt}{\@notationplain}}
\def\@notationplain{\paragraph{Notation} \it}
\def\@notationopt[#1]{\paragraph{Notation \normalfont #1}  \it}

\newenvironment{convention}[1][]
{\paragraph{Convention} {\normalfont #1}}
{\vspace{\the\theorempostskipamount}}
\def\convention{\@ifnextchar[{\@conventionopt}{\@conventionplain}}
\def\@conventionplain{\paragraph{Convention} \it}
\def\@conventionopt[#1]{\paragraph{Convention \normalfont #1}  \it}

\newenvironment{warning}[1][]
{\paragraph{Warning} {\normalfont #1}}
{\vspace{\the\theorempostskipamount}}
\def\warning{\@ifnextchar[{\@warningopt}{\@warningplain}}
\def\@warningplain{\paragraph{Warning} \it}
\def\@warningopt[#1]{\paragraph{Warning \normalfont #1}  \it}

\newenvironment{de'finition}[1][]
{\paragraph{Définition} {\normalfont #1} \it}
{\vspace{\the\theorempostskipamount}}

\newcommand{\thmendspace}{\vspace{\the\theorempostskipamount}}
\newlength{\negvcorr}
\newlength{\mnegvcorr}
\makeatother

\newenvironment{proof}[1][Proof]{\bigskip\noindent \textit{#1.~}}
{\hfill $\Box$}

\newcommand{\qed}{\hfill $\Box$}

\newenvironment{noindparaenum}
{\setlength{\parindent}{0mm}
\begin{asparaenum}}
{\end{asparaenum}}

\newcommand{\dialogue}[1]%
{\textcolor{red}{$\ulcorner$ #1 %\hfill
 $\lrcorner$}}

\newcommand{\aparte}[1]%
{$\ulcorner$ #1 %\hfill
 $\lrcorner$}

\newlength{\parindentmem}
\makeatletter

\def\@removefromreset#1#2{\let\@tempb\@elt
   \def\@tempa#1{@&#1}\expandafter\let\csname @*#1*\endcsname\@tempa
   \def\@elt##1{\expandafter\ifx\csname @*##1*\endcsname\@tempa\else
         \noexpand\@elt{##1}\fi}%
   \expandafter\edef\csname cl@#2\endcsname{\csname cl@#2\endcsname}%
   \let\@elt\@tempb
   \expandafter\let\csname @*#1*\endcsname\@undefined}

\@addtoreset{paragraph}{section}
\@addtoreset{paragraph}{part}
\@removefromreset{paragraph}{subsubsection}
\@removefromreset{paragraph}{subsection}

\let\c@equation\c@subparagraph

\makeatother

\renewcommand{\theparagraph}{(\arabic{section}.\arabic{paragraph})}
\renewcommand{\thesubparagraph}
{(\arabic{section}.\arabic{paragraph}.\arabic{subparagraph})}

\titleformat{\part}[display]{\normalfont\Large\bfseries}%
{\partname}{0cm}{}
\titleformat{\section}[hang]{\normalfont\Large\bfseries}{}{0cm}%
{\thesection \  --\ }

\titleformat{\subsection}[runin]{\normalfont\bfseries}{}{0cm}%
{}[.]

\newcommand{\spcifnec}[1]
{\ifx#1\empty
\else ~#1.
\fi}

\titleformat{\paragraph}[runin]{\normalfont\bfseries}
{\theparagraph}{0cm}{\spcifnec}%[\kern -1em]
\titlespacing{\paragraph}{0cm}%left margin
{2.75ex plus 1ex minus .2ex}%beforesep
{.5em}%aftersep

\titleformat{\subparagraph}[runin]{\it}
{\thesubparagraph}{0cm}{\spcifnec}%[\kern -1em]
\titlespacing{\subparagraph}{0cm}%left margin
{0mm}%beforesep
{.5em}%aftersep

{\titleformat{\section}[hang]{\normalfont\large\bfseries}{}{0cm}{}
\titleformat{\subsection}[hang]{\normalfont\bfseries}{}{0cm}{}
\renewcommand{\theparagraph}{(\Alph{paragraph})}

\maketitle
}{}

\newenvironment{closing}%
{\titleformat{\section}[hang]{\normalfont\large\bfseries}{}{0cm}{}
\setlength{\itemsep}{0mm}
\small
}
{}

\makeatletter
\renewcommand\@maketitle{%
  \newpage
  \begin{center}%
  \let \footnote \thanks
    {\Large \bf \@title \par}%
    \vskip 1em%
    {\large
      \begin{tabular}[t]{c}%
        \@author
      \end{tabular}\par}%
  \end{center}%
  \par
  \vskip 1.5em}
\makeatother

\renewenvironment{abstract}
{\small \quotation
\noindent {\bf Abstract.}}{\endquotation \vskip 1cm}

\newcommand{\C}{\mathbf{C}}
\newcommand{\R}{\mathrm{R}}

\renewcommand{\P}{\mathbf{P}}
\newcommand{\PH}{\mathbf{P}\kern -.05em \mathrm{H}}
\newcommand{\A}{\mathbf{A}}
\newcommand{\D}{\mathbf{D}}

\renewcommand{\O}{\mathcal{O}}
\renewcommand{\H}{\operatorname{H}}
\newcommand{\h}{\mathrm{h}}

\newcommand{\dlbrack}{[ \kern -.4ex [}
\newcommand{\drbrack}{] \kern -.4ex ]}

\newcommand{\im}{\mathrm{im}}
\newcommand{\Cliff}{\mathrm{Cliff}}
\newcommand{\Hom}{\mathrm{Hom}}
\newcommand{\Ext}{\mathrm{Ext}}
\newcommand{\Spec}{\mathrm{Spec}}
\newcommand{\Pic}{\mathrm{Pic}}

\newcommand{\pr}{\mathrm{pr}}

\DeclareMathOperator{\Sym}{\mathrm{Sym}}

\newcommand{\PGL}{\mathrm{PGL}}

\newcommand{\vect}[1]{\langle #1 \rangle} % \span already defined,
\renewcommand{\mod}{\ \textrm{mod}\ }
\DeclareMathOperator{\cork}{cork}
\DeclareMathOperator{\coker}{coker}
\newcommand{\trsp}[1]{\vphantom{#1}^{\mathsf T\!} #1}
\newcommand{\restr}[2]{\left. #1 \right| _{#2}}
\newcommand{\dual}{^\vee}

\renewcommand{\epsilon}{\varepsilon}
\renewcommand{\geq}{\geqslant}
\renewcommand{\leq}{\leqslant}
\renewcommand{\emptyset}{\varnothing}

\newcommand{\ie}{i.e.\ } %\emph seulement en français
\newcommand{\eg}{e.g.,\ }

\newcommand{\K}{\mathcal{K}}
\newcommand{\Kprim}{\mathcal{K}^{\mathrm{prim}}}
\newcommand{\Kcan}{\mathcal{K}^{\mathrm{can}}}
\newcommand{\KC}{\mathcal{KC}}
\newcommand{\KCprim}{\mathcal{KC}^{\mathrm{prim}}}
\newcommand{\KCcan}{\mathcal{KC}^{\mathrm{can}}}
\newcommand{\cprim}{c^{\mathrm{prim}}}
\newcommand{\F}{\mathcal{F}}

\newcommand{\M}{\mathcal{M}}

\newcommand{\FS}{\mathcal{FS}}
\newcommand{\bx}{\mathbf x}
\newcommand{\bef}{\mathbf f}
\newcommand{\ber}{\mathbf r}
\newcommand{\bh}{\mathbf h}
\newcommand{\CT}{\mathfrak{L}}% fu {\mathfrak C _{\mathfrak T}}

\begin{document}

\renewcommand{\O}{\mathcal{O}}

\setdefaultenum{(i)}{}{}{}
%\setdefaultitem{---}{}{}{}

\title{Wahl maps and extensions of canonical curves and $K3$ surfaces}
\author{Ciro Ciliberto, Thomas Dedieu, and Edoardo Sernesi}
%  Primary 14J28; Secondary 14H10, 14J10

\maketitle

\begin{abstract}
Let $C$ be a smooth projective curve 
(resp.\ $(S,L)$ a polarized $K3$ surface)
of genus $g \geq 11$,
non-tetragonal, considered in its canonical embedding in
$\P^{g-1}$
(resp.\ in its embedding in $|L|\dual \cong \P^g$). 
We prove that $C$ (resp.\ $S$) is a linear section of an arithmetically
Gorenstein normal variety $Y$ in $\P^{g+r}$, not a cone, with
$\dim(Y)=r+2$ and $\omega_Y=\O_Y(-r)$, if the cokernel of the 
Gauss--Wahl map of $C$
(resp.\ $\H^1(T_S \otimes L\dual)$)
has dimension larger or equal than $r+1$
(resp.\ $r$). This relies on previous work of
Wahl and Arbarello--Bruno--Sernesi. We provide various applications.
\end{abstract}

A central theme of this text is the \emph{extendability
problem}: Given a projective (irreducible) variety $X \subset \P^n$,
when does there 
exist a projective variety $Y \subset \P^{n+1}$, not a cone, of which
$X$ is a hyperplane section? 
Given a positive integer $r$, an \emph{$r$-extension} of $X
\subset \P^n$ is a
variety $Y \subset \P^{n+r}$ having $X$ as a section by a linear
space.
% an $r$-extension is \emph{non-trivial} when it is not a cone.
The variety $X$ is \emph{$r$-extendable} if it has an
$r$-extension that is not a cone, and \emph{extendable} if it is at
least $1$-extendable.
The following result provides a necessary condition for
extendability.

\begin{theorem}[(Lvovski \cite{lvovsky92})]%
\label{thm:lv}
Let $X\subset \P^ n$ be a smooth, projective, irreducible,
non-degenerate variety, not a quadric. Set 
\[
\alpha(X)=h^ 0(N_{X/\P^n}(-1))-n-1.
\]
If $X$ is $r$-extendable and $\alpha(X)<n$,
then $r\leq \alpha(X)$.
\end{theorem}

\bigskip
In particular, if $X$ is extendable then $\alpha (X) >0$.
The condition $\alpha (X) <n$ is necessary in Lvovski's
proof, 
and implies that $X$ is not a complete intersection.
The so-called Babylonian tower theorem, due to Barth, Van de Ven, and
Tyurin (see, \eg \cite{coanda}), asserts that complete intersections
are the only infinitely extendable varieties
among local complete intersection varieties.
As far as we know, it is an open question whether the assumption
$\alpha(X)<n$ in Theorem~\ref{thm:lv} can be replaced by the a priori
weaker condition that $X$ is not a complete intersection.

One of the objectives of this article is to 
establish that conversely, 
the condition $\alpha(X) \geq r$ is sufficient for the
$r$-extendabilty of 
canonical curves (Theorem~\ref{thm:cds2}) and $K3$ surfaces
(Theorem~\ref{thm:cds2K3}).

\bigskip
Let $C \subset \P^{g-1}$ be a canonical curve of genus $g$.
%(see subsection~\ref{s:notation}). 
We consider its \emph{Wahl map} 
\[
\Phi_C:\ \textstyle {
\sum_i s_i\wedge t_i \in \bigwedge^ 2 \H^0 (C,\omega_C)}
\ \longmapsto \ 
\sum_i (s_i\cdot dt_i-t_i\cdot ds_i) \in \H^0(C, \omega_C ^{\otimes 3}),
\]
see, %subsection~\ref{s:gaussian}.
\eg \cite{chm88}.
The invariant $\alpha (C)$ in Theorem~\ref{thm:lv} equals the corank
$\cork (\Phi_C)$ of the Wahl map, see Lemma~\ref{l:wahl}.
Thus, as a particular case of Theorem~\ref{thm:lv}, one has that if a
smooth curve $C$ sits on a $K3$ surface then $\Phi_C$ is
non-surjective. This was originally proved by Wahl \cite{wahl87},
% \cite[Thm.~5.9]{wahl87}
using the deformation theory of cones.
Beauville and Mérindol \cite{beauville-merindol} gave another
proof, based on the observation that for a smooth and irreducible
curve $C$ sitting on a $K3$ 
surface $S$, the surjectivity of $\Phi_C$ implies the splitting of the
normal bundle exact sequence,
\begin{equation*}
0 \to T_C \to \restr {T_S} C \to N_{C/S} \to 0.
\end{equation*}
This introduced the idea, explicit in Voisin's article
\cite{voisin-acta}, that the elements of $(\coker (\Phi_C))\dual$
(or rather of $\ker (\trsp \Phi_C)$)
should be interpreted as \emph{ribbons}, or infinitesimal surfaces,
embedded in $\P ^g$ and extending $C$: see %subsection~\ref{s:ribb}.
Section~\ref{S:ribbons}.

The following statement is a first converse to Theorem~\ref{thm:lv},
and a central element of the proofs of our Theorems \ref{thm:cds2}
and \ref{thm:cds2K3}. 
% It is the combination of results
% by Wahl and Arbarello--Bruno--Sernesi.
\begin{theorem}
[(Wahl \cite{wahl97}, Arbarello--Bruno--Sernesi \cite{abs1})]
\label{t:w+abs}
Let $C$ be a smooth curve of genus $g \geq 11$, and Clifford
index $\Cliff(C)>2$. Every ribbon $v \in \ker (\trsp \Phi_C)$
may be integrated to (\ie is contained in)
a surface $S$ in $\P^{g}$ having the canonical
model of $C$ as a hyperplane section.
\end{theorem}

\bigskip
Note that if $v \neq 0$, then
the surface $S$ is not a cone as
only the trivial ribbon may be integrated to a cone.
Conversely, we observe that actually unicity holds in Theorem~\ref{t:w+abs}
(see \ref{sp:unicita} and Remark~\ref{r:unicita}): up to isomorphisms,
given a ribbon $v \in \ker (\trsp \Phi_C)$, the surface $S$
integrating it in $\P^{g}$ is unique.
For $v=0$, this is the content of the aforementioned
theorem of Wahl and Beauville--Mérindol, see \ref{p:unic=>wahl}.

We prove a statement for $K3$ surfaces analogous to
Theorem~\ref{t:w+abs} (Theorem~\ref{t:intK3}).

\bigskip
Theorem~\ref{t:w+abs} provides a characterization of those curves
having non-surjective Wahl map in the range $g\geq 11$ and $\Cliff>2$.
Wahl \cite[p.~80]{wahl90} suggested to
study the stratification of the moduli space of curves by the corank
of the map $\Phi_C$: This is done in our Theorem~\ref{thm:cds2} to the
effect that, in the same range, the curves with $\cork
(\Phi_C) \geq r$ are those which are $r$-extendable.

We give various applications of our results, in particular to the
smoothness of the fibres of the forgetful map which to a pair $(S,C)$
associates the modulus of $C$,
where $S \subset \P^g$ is a $K3$ surface and $C$ is a canonical curve
hyperplane section of $S$ 
(Theorem~\ref{thm:cds}). The same result is proven for the analogous
map on pairs $(V,S)$ where $V$ is a Fano threefold and $S$ a smooth
anticanonical section of $V$ (Theorem~\ref{thm:cdsK3}); in this case,
this is closely related to Beauville's main result in
\cite{beauville-fano}. We also answer a question
asked in that article, see Proposition~\ref{pr:q:Beauville}.

We study the Wahl maps and extensions of (the smooth models of)
plane curves with up to nine ordinary singularities, and apply this to
solve a conjecture of Wahl \cite[p.~80]{wahl90} in the particular case
of Del Pezzo surfaces (Proposition \ref{pr:wahl-conj}).

We give a detailed account of our
results in \S~\ref{S:results}.
The substance of the proofs, together with the technical material needed
for them, is contained in \S~\ref{S:cohom}--\ref{S:plane}. 
More information on the
organization of the paper is given along \S~\ref{S:results}.

\bigskip
\par\noindent
\textbf{Thanks.}
We thank (in alphabetical order)
Gavin Brown,
Cinzia Casagrande,
Andreas Höring,
Andreas Knutsen,
and Serge Lvovski,
for their kind and inspiring answers to our questions.
We also thank the anonymous referee for valuable comments and
suggestions.

% \bigskip
% \par\noindent
% \textbf{Acknowledgements.}
% CC and ThD were membres of project FOSICAV when carrying out the
% research which lead to the present article;
% this project has received funding from the European Union's Horizon
% 2020 research and innovation programme under the Marie
% Sk{\l}odowska-Curie grant agreement No~652782.

% {\small
% \setlength{\cftaftertoctitleskip}{0.2cm}
% \renewcommand{\cfttoctitlefont}{\normalfont\large\bfseries}
% \setlength{\cftbeforesecskip}{0.05cm}
% %\setcounter{tocdepth}{1}
% \tableofcontents
% }

% \section{Preliminary definitions}
% \label{S:prelim}

% \subsection{Notation and conventions}
% \label{s:notation}

\section{Notation and conventions}
\label{S:prelim}

We work over the field $\C$ of complex numbers. All varieties, \eg curves,
surfaces, threefolds, etc., are assumed
to be integral and projective.
% \smallskip
% A \emph{canonical curve} is a smooth, non-hyperelliptic curve of genus
% $g\geqslant 3$, embedded in $\P^ {g-1}$ by its complete
% canonical series.
\par A \emph{$K3$ surface} is a smooth complete surface $S$ such that 
$\omega_S \cong \O_S$ and $\H^1(S,\O_S)=0$; 
a \emph{$K3$ surface with canonical singularities}, or 
\emph{$K3$ surface possibly with $ADE$ singularities}, or
\emph{possibly singular $K3$ surface},
 is a surface
with canonical singularities whose minimal desingularisation is a $K3$
surface.
%\par 
A \emph{fake $K3$ surface} 
is a non-degenerate, projective surface in $\P^ g$,
not a possibly singular $K3$ surface,
having as a hyperplane section a smooth, canonical curve 
$C \subset \P^{g-1}$ of genus $g \geq 3$. 
% see subsection~\ref{s:epema}.

The \emph{Clifford index} $\Cliff(S,L)$ of a polarized
$K3$ surface $(S,L)$ is the Clifford index of any smooth curve $C \in
|L|$; by \cite {green-lazarsfeld}, this does not depend on the choice
of $C$.

We denote by:\\
\begin{inparaitem}
\item $\M_g$ the moduli stack of smooth curves of genus $g$;\\
\item  $\K_g$ (resp.\ $\Kprim_g$) the moduli stack of polarised
  (resp.\ primitively polarised) $K3$ surfaces of genus $g$,
\ie pairs $(S,L)$ with $S$ a smooth $K3$ surface, and $L$ an ample,
globally generated
%big and nef
(resp.\ and primitive) line bundle 
 on $S$ with $L^ 2=2g-2$; \\
\item $\Kcan_g$ 
the moduli stack of polarised, possibly singular, $K3$ surfaces of
genus $g$,
\ie pairs $(S,L)$ with $S$ a $K3$ surface with canonical singularities,
% whose minimal desingularisation is a $K3$ surface, 
and $L$ an ample,
globally generated
line bundle on $S$ with $L^ 2=2g-2$,
see \cite[5.1.4]{huybrechts};
\\
\item $\KC_g$ (resp.\ $\KCprim_g$, $\KCcan_g$) 
the moduli stack of pairs $(S,C)$
  with $C$ a smooth curve on $S$ and $(S,\O_S(C))\in \K_g$
(resp.\ $\Kprim_g$, $\Kcan_g$);\\
\item $\F_g$ the moduli stack of Fano threefolds $V$ of genus $g$
(not necessarily of index $1$), \ie
smooth varieties $V$ with $-K_V$ ample, and $K_V^ 3=2-2g$;\\
\item $\FS_g$ the moduli stack of pairs $(V,S)$ with $V\in \F_g$ and
$S\in |-K_V|$ a smooth surface, so that 
$(S,\restr {-K_{V}} S)\in \K_g$;\\
\item $\K_g^ \R$, where $\R=(R,\lambda)$ consists of a lattice $R$ and a
distinguished element  $\lambda\in R$ with $\lambda^ 2=2g-2$, 
the moduli stack of \emph{$\R$-polarised} $K3$ surfaces,
\ie polarised $K3$ surfaces $(S,L)$ together with a
fixed embedding of $R$ as a primitive sublattice of $\Pic(S)$, sending
$\lambda$ to the class of $L$;\\
\item $\KC_g^ \R$, with $\R$ as above,
the moduli stack of pairs $(S,C)$ with $S$ an
$\R$-polarised $K3$ surface and $C$ a smooth curve on $S$ in the class
$\lambda$;\\
\item $\F^ \R_g$ and $\FS^\R_g$ the stacks of Fano varieties analogous
  to $\K_g^ \R$ and $\KC_g^ \R$;\\
\item $c_g: \KC_g\to \M_g$, $\cprim_g: \KCprim_g\to \M_g$, and $c^
  \R_g: \KC^ \R_g\to \M_g$ the forgetful maps;\\ 
\item $s_g: \FS_g\to \K_g$ and $s_g^ \R: \FS^ \R_g\to \K^ \R_g$ the
forgetful maps.
\end{inparaitem}

\section{Main results}
\label{S:results}

\subsection{Canonical curves}
\label{s:results-can}

Our first result is the following converse to Lvovski's
Theorem~\ref{thm:lv} for canonical curves.

\begin{theorem}\label{thm:cds2} Let $C$ be a smooth genus $g$ curve 
with Clifford index $\Cliff(C)>2$,
and $r$ a non-negative integer.
We consider the following two propositions:
\begin{noindparaenum}
\item \label{cork} $\cork(\Phi_C)\geq r+1$;
\item \label {ext} there is an arithmetically Gorenstein normal variety
$Y$ in $\P^{g+r}$, not a cone, with $\dim(Y)=r+2$,
$\omega_Y=\O_Y(-r)$, which has a canonical image of $C$ as a section
with a  $(g-1)$-dimensional linear subspace of $\P^{g+r}$
(in particular, $C\subset \P^{g-1}$ is $(r+1)$-extendable).
\end{noindparaenum}
\par If $g \geq 11$, then \eqref {cork} implies \eqref {ext}.
Conversely, if $g \geq 22$ and the canonical image of $C$ is a
hyperplane section of some smooth $K3$ surface in $\P^g$, then
\eqref{ext} implies \eqref{cork}. 
\end{theorem}

\paragraph{}
\label{p:univ-ext}
Actually, we prove more than
Theorem~\ref{thm:cds2}, see Corollary~\ref{c:univ-ext}. Let $C$ be a
smooth
curve of genus $g \geq 11$ with $\Cliff (C) >2$, and let
$r := \cork (\Phi_C) -1$.
\subparagraph{}\smallskip
\label{sp:univ-ext}
{\it There is an arithmetically Gorenstein normal variety $X$ of
dimension $r+2$ in $\P^{g+r}$, not a cone,
containing a canonical image $C_0$ of $C$ as a section by a
linear $(g-1)$-space, and satisfying the following 
property: for all $[v] \in \P(\ker(\trsp \Phi_C))$, there is a unique
section of $X$ by a linear $g$-space containing a ribbon over $C_0$
belonging to the isomorphism class $[v]$.
}\\[\the\smallskipamount]
(See Section~\ref{S:ribbons} for background on ribbons and their
relation with the Wahl map $\Phi_C$).
\subparagraph{}\smallskip
\label{sp:unicita}
{\it For all $[v] \in \P(\ker(\trsp \Phi_C))$, there is a
\emph{unique} (up to projectivities pointwise fixing $C$, see
Remark~\ref{r:unicita}) surface $S \subset \P^g$ containing 
a ribbon over a canonical model of $C$ in the isomorphism class $[v]$.
}
\subparagraph{Definition}\smallskip%
\label{def:universal}
We say that an extension $X$ of $C_0$ as in \ref{sp:univ-ext} is
\emph{universal}. By \ref{sp:unicita}, a universal extension
of $C_0$ has as linear sections all possible
surface extensions of $C_0$ but cones.

\bigskip
No matter the genus, Lvovski's Theorem~\ref{thm:lv} tells us that
\eqref{ext} of Theorem~\ref{thm:cds2} implies the inequality 
\[
\cork(\Phi_C)\geq \min(g-1,r+1).
\]
%\par\medskip
When $r=0$, '$\eqref{ext} \Rightarrow \eqref{cork}$' in
Theorem~\ref{thm:cds2} was proved by Wahl
\cite{wahl87} and later independently by Beauville and Mérindol
\cite{beauville-merindol}, and
'$\eqref{cork} \Rightarrow \eqref{ext}$' is Theorem~\ref{t:w+abs}
by Wahl and Arbarello--Bruno--Sernesi. 
To prove $\eqref{cork} \Rightarrow \eqref{ext}$ for arbitrary $r$, we
show that Wahl's extension construction \cite[Theorem~7.1]{wahl97}
(the requirements of which are met thanks to \cite[Theorem~3]{abs1})
works in families, see \S~\ref{S:extension}.

Statement~\ref{sp:unicita} is implicitly contained in the proof of
\cite[Theorem 7.1]{wahl97} as we observe in Remark~\ref{r:unicita},
although it apparently remained unnoticed so far. 

\paragraph{}
\label{p:unic=>wahl}
If $g \geq 11$ and $\Cliff (C) >2$,
the unicity of the extension of the ribbon $v=0$ in
Theorem~\ref{t:w+abs} (see Remark~\ref{r:unicita})
tells us that the cone over a canonical model of $C$ is
the only surface in $\P^g$ containing the trivial ribbon over
$C$. Thus, if $C$ sits on a $K3$ surface, the ribbon over $C$ defined
by $S$ is non-trivial, hence $\Phi_C$ is not surjective: 
this is the theorem of Wahl and Beauville--Mérindol, though a priori
only for curves of genus $g \geq 11$ and Clifford index $\Cliff (C)
>2$; the remaining cases can be dealt with directly:
curves with $g <10$ or $\Cliff (C) \leq 2$ all
have non-surjective Wahl map \cite{wahl90,cm90,cm92,brawner},
and curves of genus $10$ sitting on a $K3$ have non-surjective Wahl
map by \cite{cukierman-ulmer}.

\begin{proof}  [Proof of Theorem~\ref{thm:cds2}]
The fact that \eqref {cork} implies \eqref {ext} provided $g \geq 11$
is the content of Corollary \ref {cor:ext}.
The converse implication is given by Lvovski's Theorem~\ref {thm:lv}
as follows.
Identify $C$ with its canonical model in $\P^{g-1}$. Then 
$\alpha(C)=h^ 0(N_{C/\P^ {g-1}}(-1))-g$ equals $\cork (\Phi_C)$ by
Lemma~\ref{l:wahl}, so one has $\alpha (C) \leq 20$ by
Corollary~\ref{c:bnd-alpha}. It follows that the assumption $\alpha
(C) < g-1$ holds if $g \geq 22$; in this case, Theorem~\ref {thm:lv}
says that \eqref{ext} implies \eqref{cork}.
\end{proof}

\bigskip
We obtain the bound $\alpha (C) \leq 20$ used in the above proof of
Theorem~\ref{thm:cds2} as a corollary of
Proposition~\ref{pr:finite}. The latter Proposition also has the
following consequence, of independent interest.

\begin{proposition}[(Corollary~\ref{cor:finite2})]%
\label{pr:modmap-finite}
Let $(S,C)\in \KCcan_g$ with $g\geq 11$ and $\Cliff(C)>2$.
There are only finitely many members $C'$ of $|\O_S(C)|$ that are
isomorphic to $C$.
\end{proposition}

\medskip
In fact it follows from the arguments in the proof of
Corollary~\ref{cor:finite2} that if $\cork(\Phi_C)=1$
(which happens for instance if $g>37$, see
Corollary~\ref{c:bndedness} below), then for all $C' \in |\O_S(C)|$
isomorphic to $C$ there exists an automorphism of the polarized surface
$(S,\O_S(C))$ taking $C$ to $C'$
(because in this case the two curves $C$ and $C'$ have the same ribbon
in $S$).
In particular, if the automorphism group of $S$ is trivial (which
happens for instance if $S$ has Picard number $1$), then the smooth
members of $|\O_S(C)|$ are pairwise non-isomorphic.

%\bigskip
% By the way, let us point out that there is a problem with
% \cite[Proposition~1.2 and Corrigendum]{ck14} by the first-named author
% and Knutsen, which would assert something stronger than the previous
% Proposition~\ref{pr:modmap-finite}. We report about this in
% Remark~\ref{r:corr-sbagl}.

\bigskip
The following is a consequence of \ref{p:univ-ext}:

\begin{corollary}
Let $C$ be a smooth curve of genus $g \geq 11$ with
$\Cliff (C) >2$. 
The curve $C$ cannot sit on two $K3$ surfaces $S$ and $S'$ such that
its respective classes in $\Pic (S)$ and $\Pic (S')$ have different
divisibilities.
\end{corollary}

%\vspace{\negvcorr}
\begin{proof}
By \ref{p:univ-ext}, all extensions of the canonical
model of $C$ are packaged together in an irreducible family. The
Corollary thus follows from the fact that the divisibility of $[C]$ in
$\Pic (S)$ is a topological character, hence constant under
deformations of the pair $(C,S)$.
\end{proof}

%\paragraph{}
\bigskip
Next, we study the ramification of the forgetful map 
$c_g: \KC_g \to \M_g$. 
To put our results in perspective, recall that 
\[
\dim (\KC_g) - \dim (\M_g) = (19+g)-(3g-3) = 2(11-g).
\]
The primitively polarised case has been classically studied:
for $g\geq 11$, the map $\cprim_g$ is birational onto its image if
$g\neq 12$,
whereas its generic fibre is irreducible of dimension $1$ if $g=12$
(\cite[\S~5.3]{clm93} and \cite{mukai2});
for $g \leq 11$, the map $\cprim_g$ is dominant if $g \neq 10$
\cite{mukai1}, and onto a hypersurface of $\M _{10}$ if $g=10$
\cite{cukierman-ulmer}, with irreducible general fibre in any case
\cite{cm90,clm93}.
The non-primitively polarised cases have been studied in
\cite{cfgk-adv, kemeny} where it is shown that, if $g \geq 11$ then
$c_g$ is generically finite in all but possibly finitely
many cases.
\par
It turns out that in the range $g \geq 11$, the map $c_g$ has smooth
fibres over the locus of curves with Clifford index greater than $2$.

\begin{theorem}\label{thm:cds}
Let $(S,C)\in \KC_g$ with $g\geq 11$ and $\Cliff(C)>2$. Then
\[
\dim(\ker (dc_g)_{(S,C)})=\dim(c_g^ {-1}(C))=\cork (\Phi_C)-1.
\]
\end{theorem}
\vspace{\mnegvcorr}

Over curves with 
% Clifford index smaller or equal than $2$ 
$\Cliff(C)\leq 2$
the situation is more complicated, if only because then the spaces 
$\H^0(N_{C/\P^{g-1}}(-k))$, $k \geq 2$, don't necessarily vanish 
(equivalently the higher Gaussian maps 
$\Phi_{\omega_C ^{\otimes k}, \omega_C}$,
$k \geq 2$, are not necessarily surjective \cite{wahl88}),
contrary to what happens when $\Cliff (C)>2$, compare
Lemma~\ref{l:fighissimo}. 
See \cite[Cor.~4.4]{cm90} for the situation over the general curve of
genus $g \leq 6$.

\begin{remark}
Beauville \cite[Sec.~5]{beauville-fano} observed that the map $c_g$ is
not everywhere unramified, as it has positive dimensional fibres at
those points $(S,C)$ such that $S$ is an anticanonical divisor of some
smooth Fano threefold $V$.
Theorems~\ref{thm:cds} and \ref{thm:cds2} together say that,
in the range of application of Theorem~\ref{thm:cds}, all the
ramification of $c_g$ is accounted for by this phenomenon.
\par This reflects the fact that for $g \leq 12$, $g \neq 11$, the
positive dimensionality of the generic fibre of $\cprim_g$ is explained by
the existence of Fano varieties with coindex $3$ and Picard number $1$ 
(see, e.g., \cite[Chap.~5]{encyclopedia}).
\end{remark}

\begin{proof}[Proof of Theorem~\ref {thm:cds}]
%In the subsequent sections we prove intermediate results which
%provide
We have
the following chain of (in)equalities:
\begin{alignat*}{2}
\cork (\Phi_C)-1 &\leq   \dim(c_g^ {-1}(C))
&\qquad& \text{by Corollary~\ref {cor:palese-vero}} \\
&\leq \dim(\ker (dc_g)_{(S,C)})
&\qquad& \text {obviously}\\
&= h^ 1(T_S(-1))
&\qquad& \text{by Lemma~\ref {l:sernesi}} \\
&\leq \cork (\Phi_C)-1
&\qquad& \text{by Proposition~\ref {prop:ineq}.}
\end{alignat*}
\end{proof}

\bigskip
The following result is a straightforward but noteworthy consequence
of the proof of Theorem~\ref{thm:cds}. It says in particular that the
corank of the Wahl map is the same for all smooth hyperplane sections
of a given $K3$ surface. 
% We refer to the next subsection (right before
% Theorem~\ref{t:intK3}) for the definition of the Clifford index of a
% polarized $K3$ surface.

\begin{corollary}
\label{c:intrsting-rmk}
Let $(S,L) \in \K_g$, 
and assume that $g \geq 11$ and $\Cliff (S,L) >2$.
For every smooth member $C \in |L|$, one has
\[
\cork (\Phi_C) = h^1 (T_S(-1)) +1.
\]
\end{corollary}
% \vspace{\negvcorr}
% \vspace{\negvcorr}

\paragraph{}
\label{p:prokhorov}
It is a known fact that a threefold $V$ in
$\P^{g+1}$ having as hyperplane section a $K3$ surface $S$,
possibly with $ADE$ singularities, is an arithmetically Gorenstein 
%(\ie projectively normal and Gorenstein) 
normal Fano threefold with canonical
singularities, see Corollary~\ref{C:ext}.
%{cor:ext} and \ref{cor:ext2}.
% If moreover $(S,\restr{-K_V}S)$ is primitively
% polarized, then $V$ has index $1$.
Consequently, in the setting of Theorem~\ref{thm:cds2}, 
if we assume in addition that there exists an extension of $C$ to a
surface $S$ with at worst $ADE$ singularities
(so that $S$ is a $K3$ surface, possibly singular),
then the sections of $Y$ with linear subspaces of dimension $g+1$
containing $S$ are Fano threefolds of
genus $g$, with canonical singularities.
We may thus use the boundedness of Fano varieties to derive 
the following corollary from our previous results.

\begin{corollary}
\label{c:bndedness}
Let $C$ be a smooth curve of genus $g>37$, and Clifford index $\Cliff
(C) >2$. If the canonical model of
$C$ is a hyperplane section of a $K3$ surface $S$,
possibly with $ADE$ singularities,
then
$\cork (\Phi_C) = 1$.
\end{corollary}
% \vspace{\negvcorr}
%\vspace{\negvcorr}

\begin{proof}
If $C$ is a hyperplane section of a $K3$ surface $S$
and $\cork (\Phi_C) > 1$, then by
Corollary~\ref{C:ext} there is an
arithmetically Gorenstein Fano threefold of genus
$g$, with canonical singularities, and having $C$ as a curve section.
By \cite[Thm.~1.5]{prokhorov} all such threefolds have genus $g \leq
37$.
\end{proof}

\paragraph{Remark}
Based on the above statement, one may be tempted to speculate that
all smooth curves $C$ of genus $g>37$ with $\Cliff (C)>2$ have
$\cork (\Phi_C) \leq 1$; this is not true.

If one drops the assumption that the curve $C$ lies on a $K3$ surface
in Corollary~\ref{c:bndedness}, one has to deal with the possibility
that all surface extensions of the curve $C$ may have singularities
worse than $ADE$ singularities. In such a situation, threefolds
extending $C$ are no longer Fano, and there is no boundedness result
in this case.

As a matter of fact, plane curves provide examples of curves of
arbitrarily large genus, having Clifford index greater than $2$, and
for which the Wahl map has corank $10$ \cite[Thm.~4.8]{wahl90}.

\paragraph{}
\label{p:plane-intro}
In Section~\ref{S:plane} 
we study the extensions of the canonical models of 
%curves having a (possibly singular) plane model, 
the normalizations of plane curves,
continuing a long story contributed to by numerous authors
(see at least \cite{epema, abs1} and the references therein).
% This provides examples of curves
% of arbitrarily high genus, with Clifford index greater than $2$,
% that have Gauss maps of rather high corank and are indeed extendable,
% but nevertheless have no smooth extension.
Such surface extensions are rational, hence not $K3$,
and have indeed an elliptic (in general non-smoothable) singularity.
\par
We give an explicit construction of the universal extensions of
such curves.
These extensions are not Fano, and provide an unbounded family of
``fake Fano'' varieties,
% mimicking the definition of fake $K3$ surfaces,
% given in subsection~\ref{s:notation}, 
% a \emph{fake Fano variety} is
i.e.\ 
irreducible varieties $X$ of dimension $r+2$ in $\P^ {g+r}$
($r>0$), with non-canonical singularities,
and with curve sections canonical curves of genus $g$.
Whereas fake $K3$ surfaces are fairly well understood (for
instance, there is a classification \cite{epema}), understanding fake
Fano varieties is a wide open problem.

\par We use the precise relation between extensions and cokernel of
the Gauss map to 
% obtain a lower bound on the corank of the Gauss map
% of plane curves with up to nine ordinary singular points. This gives a
prove a conjecture of Wahl
\cite[p.~80]{wahl90} in the case of Del Pezzo surfaces, see
Proposition~\ref{pr:wahl-conj}; 
the case of the projective plane was already
handled in \cite{wahl90}.

% \bigskip
% The following remark belongs to the same circle of ideas, and also
% relies on the fact that Fano threefolds are bounded, as it is based on
% the classification of smooth Fano threefolds.

\begin{remark}
\label{rk:classif}
\emph{All canonical curves in smooth Fano threefolds $V$ with 
Picard number $\rho (V) \geq 2$ are Brill--Noether special.}
\par\medskip
 More precisely, we claim that if $V$ is a smooth Fano threefold
$V$ with $\rho (V) \geq 2$, then the smooth curves in $V$
complete intersections of two elements of $|-K_V|$ are Brill--Noether
special.  To see this, one has to consider one by one all the elements
in the list \cite[Chap.~12]{encyclopedia}, and check that in each case
there is a 
line bundle on $V$ that gives Brill--Noether special linear series 
%(\ie linear series $g^r_d$  with $\rho(g, r, d)<0$) 
on the canonical curves contained in $V$. 
We do not dwell on this here. 
% \par By the way, note that the main result in
% \cite{green-lazarsfeld} asserts that it is a necessary
% condition for a curve $C$ as above to be Brill--Noether special, that
% there is a line bundle on $V$ giving a special linear series on $C$.
\end{remark}

Since all smooth Fano threefolds with Picard number $1$ have genus $g
\leq 12$, $g \neq 11$,
Remark~\ref{rk:classif} together with \ref{p:prokhorov} leads to the
following question.

\begin{question}
\label{q:BN-cork>1}
Does there exist any Brill--Noether general curve of genus $g \geq
11$, $g \neq 12$, such that
$\cork (\Phi_C) > 1$?
\end{question}

\medskip
We cannot answer this question so far, for the following two reasons:
(i) as far as we know, no classification of Fano threefolds with
Gorenstein canonical singularities is available, 
and (ii) there is the possibility that all surface extensions of
a given curve have singularities worse than $ADE$ singularities.

Note however that the singularities of a surface extension of a
Brill--Noether--Petri general curve cannot be too bad: it is proven in
\cite{abs1} that such an extension is smoothable to a $K3$ surface
if $g \geq 12$.

The so-called \emph{du Val curves} 
(a particular instance of the curves we study in
\S~\ref{S:plane},
see \cite{abfs})
are an interesting example with regard to this problematic.
% given an integer $g>1$, a du Val curve is
% the normalisation of a plane curve of degree $3g$, with eight points
% of multiplicity $g$ and one point of multiplicity $g-1$, all nine
% lying on a smooth cubic, and no other (even infinitely near)
% singularities; such a curve has genus $g$. 
%, see \ref{p:plane-intro} above.
Under suitable generality assumptions, 
a du Val curve is Brill--Noether--Petri
general \cite{treibich,abfs};
its Wahl map has corank 
$1$ if $g$ is odd \cite{ab}, and is unknown otherwise; this leaves
Question~\ref{q:BN-cork>1} open. 
Note that when $g$ is odd, the canonical model of a general du Val
curve has a unique surface extension, which is a rational surface with
a smoothable elliptic singularity (see
Proposition~\ref{pr:plane-curves}).

\paragraph{}
\label{p:plane-curves}
In Corollary~\ref{c:bnd-alpha}, we prove as a consequence of
Proposition~\ref{pr:modmap-finite} (Proposition~\ref{pr:finite})
%Remark~\ref{r:maxmod} 
that $\cork(\Phi_C) \leq 20$ for any 
smooth curve of genus $g \geq 11$ and $\Cliff (C) >2$ lying on a
smooth $K3$ surface.
% (we have already mentioned this upper bound before, as it is used in
% the proof of Theorem~\ref{thm:cds2}).
We suspect that this bound is far from being sharp.
% \par
% The corank of the Wahl map is known for curves having a plane model,
% smooth \cite{wahl90} or with various singularities \cite{kang, ab,
% edoardo-DP1}, the maximal known value being $10$, which is reached for
% smooth plane curves.
\par
The corank of the Wahl map of a general tetragonal curve of genus $g
\geq 7$ equals $9$ \cite{brawner}. 
On the other hand, the corank of the Wahl map of a hyperelliptic
(resp.\ trigonal) curve of genus $g$ is $3g-2$ \cite{wahl90, cm92}
(resp.\ $g+5$ \cite{brawner,cm92}).
Note that \cite{wahl90, cm92} also assert that $3g-2$ is the maximal
possible value for the corank of the Gauss map of a curve of genus
$g$, and is attained only for hyperelliptic curves.
\begin{question}
Does there exist a universal, genus independent, bound on $\cork
(\Phi_C)$ for curves $C$ with Clifford index $\Cliff (C) > 2$?
\end{question}

\subsection{$K3$ surfaces}
\label{s:results-K3}

Generally speaking, the results about canonical curves discussed above
pass to their smooth extensions in $\P^g$, namely $K3$ surfaces.
First of all, we prove the following result for $K3$ surfaces,
analogous to Theorem~\ref{t:w+abs}.
% in its complete form including the
% unicity statement~\ref{sp:unicita}.
Given $(S,L) \in \K_g$, we consider $S$ in its embedding in
$\P^g=|L|\dual$. 
% By definition, the Clifford index of a polarized $K3$ surface $(S,L)$
% is the Clifford index of any smooth curve $C \in |L|$; 
% by \cite {green-lazarsfeld}, this does not depend on the choice of $C$.
\begin{theorem}
\label{t:intK3}
Let $(S,L) \in \K_g$ be a polarized $K3$ surface of genus
$g \geq 11$, such that $\Cliff(S,L)>2$. Every ribbon $v \in
\H^1(T_S \otimes L\dual)$ may be integrated to a unique threefold $V$ in
$\P^{g+1}$, up to projectivities.
\end{theorem}

\bigskip
As in Theorem~\ref{t:w+abs}, if $v \neq 0$ in the above statement,
then $V$ is not a cone. In particular, a polarized $K3$ surface
$(S,L)$ with $g \geq 11$ and $\Cliff(S,L)>2$
lies on a Fano threefold (with canonical Gorenstein
singularities, see \ref{p:prokhorov}) if and only if
$\H^1(T_S \otimes L\dual) \neq 0$.

The necessary background on ribbons is given in
\S~\ref{S:ribbons}, and the proof of Theorem~\ref{t:intK3}
in \S~\ref{S:intK3};
it relies on the existence of a universal extension for canonical
curves (see \ref{p:univ-ext}) and on Corollary~\ref{c:intrsting-rmk}.
Next, Theorems~\ref{thm:cds2K3} and \ref{thm:cdsK3} are the exact
analogues for $K3$ surfaces of Theorems~\ref{thm:cds2} and
\ref{thm:cds}.

\begin{theorem}
\label{thm:cds2K3}
Let $(S,L) \in \K_g$ be a polarized $K3$
surface 
with Clifford index $\Cliff(S,L)>2$. We consider the following two
propositions: 
\begin{noindparaenum}
\item \label{corkK3} $h^1\bigl( T_S \otimes L\dual \bigr) \geq r$;
\item \label {extK3} there is an arithmetically Gorenstein normal
variety $X$ in $\P^ {g+r}$, with $\dim(X)=r+2$, $\omega_X=\O_X(-r)$,
$X$ not a cone, having the image of $S$ by the linear system $|L|$ as
a section with a linear subspace of dimension $g$.
\end{noindparaenum}
\par If $g \geq 11$, then \eqref {corkK3} implies \eqref {extK3}.
Conversely, if $g \geq 22$
then
\eqref{ext} implies \eqref{cork}. 
\end{theorem}

\bigskip
% Under the assumptions of Theorem~\ref{thm:cds2K3},
% we see $S$ as embedded in $\P^g$ by the complete linear system $|L|$;
% then, 
By Lemma~\ref{l:wahl1},
one has
\[
\alpha(S) = h^1\bigl( T_S(-1) \bigr)
= h^1\bigl( T_S \otimes L\dual \bigr)
\]
Similar to the curve case,
if \eqref{extK3} holds then
$h^1\bigl( T_S(-1) \bigr) \geq \min(g,r)$
by Lvovski's Theorem~\ref{thm:lv};
in particular, $S$ is extendable if and only if 
$h^1\bigl( T_S(-1) \bigr) >0$.

\begin{proof}[Proof of Theorem~\ref{thm:cds2K3}]
Let $C$ be a smooth hyperplane section of $S \subset \P^g$: it
is a canonical curve of genus $g$ and Clifford index
$\Cliff(C)=\Cliff(S,L)$, and one has
$\cork (\Phi_C) = h^1 (T_S(-1))+1$
by Corollary~\ref{c:intrsting-rmk}.
Then Theorem~\ref{thm:cds2K3} follows at once from
Theorems~\ref{thm:cds2} and \ref{p:univ-ext}.
\end{proof}

\begin{theorem}\label{thm:cdsK3}
Let $(V,S)\in \FS_g$ with $g \geq 11$ and $\Cliff(C)>2$. Then
\[
\dim(\ker (ds_g)_{(V,S)})=\dim(s_g^ {-1}(S))
= h^1\bigl( T_S(-1) \bigr) -1.
\]
\end{theorem}
\vspace{\mnegvcorr}

% \begin{proof}%[Proof of Theorem~\ref{thm:cdsK3}]
% It is exactly the same as that of Theorem~\ref{thm:cds}.
%  with
% Theorem~\ref{t:intK3} instead of Theorem~\ref{t:w+abs}.
% One has:
% \begin{alignat*}{2}
% h^ 1(T_S(-1))-1 &\leq   \dim(s_g^ {-1}(S))
% &\qquad& \text{by Corollary~\ref{cor:palese-veroK3}} \\
% %Theorem~\ref{t:intK3} and Lemma~\ref{lem:sciocco}} \\
% &\leq \dim(\ker (ds_g)_{(V,S)})
% &\qquad& \text {obviously}\\
% &= h^ 1(T_V(-S))
% &\qquad& \text{by Lemma~\ref {l:sernesi}} \\
% &\leq h^ 1(T_S(-1))-1
% &\qquad& \text{by Proposition~\ref {prop:ineq2}.}
% \end{alignat*}
% \end{proof}

The proof of Theorem~\ref{thm:cdsK3}
is exactly the same as that of Theorem~\ref{thm:cds}.
In analogy with Corollary~\ref{c:intrsting-rmk}, it gives 
\begin{equation}
\label{eq:intrsting-rmkK3}
h^1(T_S(-1))-1 = h^ 1(T_V(-S))= h^{2,1}(V) = b_3(V)/2.
\end{equation}
Theorem~\ref{thm:cdsK3} is closely related to the following
result. 

\begin{theorem}[(Beauville, \cite{beauville-fano})]%
\label{t:beauville}
The morphism $s_g^\R: \FS_g^\R \to \K_g^\R$ is smooth and 
dominant. Its relative dimension at the point $(V,S)$ is 
$b_3(V)/2$.
\end{theorem}

% Note that, as Beauville observes, a corollary of
% Theorem~\ref{t:beauville} is that a general $K3$ surface with Picard
% lattice $\R$ is an anticanonical divisor in a Fano threefold if and
% only if $\R \cong (\Pic(V), -K_V)$ for some Fano threefold $V$
% \cite[Cor.~4.1]{beauville-fano};
% however, in most cases the map $s_g^\R$ is not surjective
% \cite[(4.2)]{beauville-fano}.

% \vspace{-4mm}

% \bigskip
% \begin{remark}
% The previous proof gives the equality
% \setcounter{equation}{\value{subparagraph}}
% \begin{equation}
% \label{eq:intrsting-rmkK3}
% h^ 1(T_S(-1))-1 = h^ 1(T_V(-S)),
% \end{equation}
% in analogy with Corollary~\ref{c:intrsting-rmk}.
% On the other hand, the two sheaves 
% $T_V(-S)$ and $\Omega^2_V$ are isomorphic, so that 
% \begin{equation}
% \label{eq:hdg-nmbrsFano}
% h^ 1(T_V(-S)) = h^{2,1}(V) = b_3(V)/2.
% \end{equation}
% \end{remark}

\bigskip
Beauville {\cite[(4.4)]{beauville-fano}}
%\label{q:beauville} 
asked:
{\itshape For those families of Fano threefolds for which $b_3=0$, the map
$s_g^\R$ is étale; is it an isomorphism
onto an open substack of $\K_g^\R$?}
We give the following answer.
% under our usual
% assumptions on the genus and Clifford index.  

\begin{proposition}
\label{pr:q:Beauville}
Let $(V,S) \in \FS_g$ be such that $g \geq 11$, 
$\Cliff (S, -\restr {K_V} S) >2$, and $b_3(V)=0$.
The fibre $(s_g^\R)^{-1}(S)$ is reduced to a point 
if and only if there is no non-trivial automorphism of $V$ induced by
a projectivity of $|-K_V|$.
\end{proposition}
%\vspace{\negvcorr}

\begin{proof}
One has
$h^1(T_S(-1)) = 1$
by \eqref{eq:intrsting-rmkK3}.
% \eqref{eq:hdg-nmbrsFano}.
From this we deduce by Theorem~\ref{t:intK3}
that up
to isomorphism $V$ is the only Fano threefold that 
may contain $S$ as an anticanonical divisor.
As a consequence, the fibre $s_g^{-1} (S)$ has cardinality greater
than $1$ if and only if there are several anticanonical divisors in $V$
isomorphic to $S$.
The latter property implies the existence of an automorphism of $V$,
induced by a projectivity of $\P^{g+1}$, that transforms one copy of
$S$ as an anticanonical divisor into another.
\end{proof}

%\subsubsection{Organization of the text}

\section{Gaussian maps and twisted normal bundles}
\label{S:cohom}

% This Section is devoted to generalities 
% % on global sections of  
% % twisted normal bundles of varieties in projective spaces, that will be
% useful throughout the text.

\paragraph{}
Let $X$ be a smooth
%irreducible, projective 
variety and $L$ a line
bundle on $X$.  We consider the multiplication map
\begin{equation}
\label{eq:mult-map}
\mu_{L,\omega_X}: \H^ 0(L)\otimes \H^ 0(\omega_X)\to 
\H^ 0(L\otimes \omega_X),
\end{equation}
whose kernel we denote by $R(L,\omega_X)$. 
If $X$ is a curve $C$, one defines the
$(L,\omega_C)$-\emph{Gaussian map} 
\begin{equation}
\label{eq:gaussian}
\textstyle
\Phi_{L,\omega_C}: R(L,\omega_C)\to \H^ 0(L\otimes \omega_C^ {\otimes 2}) 
\quad \text{by} \quad 
\sum_i s_i \otimes t_i \mapsto \sum_i (s_i \cdot dt_i- t_i\cdot ds_i),
\end{equation}
% % where the expression $\sum_i (s_i \cdot dt_i - t_i\cdot ds_i)$ is
% % defined in local coordinates and makes sense globally if $\sum_i s_i
% \otimes t_i\in R(L,\omega_C)$, see, e.g., \cite {chm88}. 
see \cite {chm88} for more details.
The map $\Phi_{\omega_C,\omega_C}$ 
restricted to $\bigwedge^2 \H^0(\omega_C)$
identifies with the Wahl map $\Phi_C$.

% In particular, if
% $L=\omega_C$, this boils down to considering the \emph{Wahl map}
% \begin{equation*}
% \Phi_C: \textstyle {\bigwedge^ 2 \H^0 (C,\omega_C)}
% \to \H^0(C, \omega_C ^{\otimes 3}),
% \quad \text{defined by} \quad 
% s \wedge t \mapsto s\cdot dt - t \cdot ds.
% \end{equation*}

\begin{lemma}[({\cite{wahl88}}, see also {\cite[Prop.~1.2]{cm90}})]
\label{l:wahl}
Let $C$ be a smooth
% irreducible projective
curve of positive genus and
$L$ a very ample line bundle on $C$. We consider $C\subset \P^
r:=\P(\H^0 (L) \dual)$. Then one has the exact sequence
\begin{equation}
\label{w-exctsq}
0 \to 
\H^0 (L) ^\vee 
\to \H^0 \bigl(N _{C/\P^{r}} \otimes L^ \vee \bigr)
\to \coker (\Phi_{L,\omega_C}) ^\vee
\to 0.
\end{equation}
\end{lemma}
\vspace{\negvcorr}
% \vspace{\negvcorr}

\paragraph{}
\label{p:identifications}
In order to state some identifications worth keeping in mind, 
valid for a curve of arbitrary genus, we
sketch the proof of Lemma~\ref{l:wahl}.
% We let $C$ be a smooth irreducible projective curve of any genus,
% $L$ a very ample line bundle on $C$, 
% %$E=\H^0(C,L)$, 
% and consider $C\subset \P^r:=\P(\H^0 (L) \dual)$.
The Euler exact sequence twisted by $L\dual$, together with Serre
duality, gives the exact sequence
% \[
% 0 \to
% E\dual \to
% \H^0 \bigl(T_{\P^r} \otimes L\dual \bigr) \to
% \H^0(L \otimes \omega_C)\dual 
% \xrightarrow{\trsp {\mu_{L,\omega_C}}}
% E\dual \otimes \H^0(\omega_C)\dual
% \to \H^1 \bigl(T_{\P^r} \otimes L\dual \bigr) 
% \to 0,
% \]
\[
0 \to
\H^0 (L)\dual \to
\H^0 \bigl(T_{\P^r} \otimes L\dual \bigr) \to
\H^0(L \otimes \omega_C)\dual 
\xrightarrow{\trsp {\mu_{L,\omega_C}}}
\H^0 (L)\dual \otimes \H^0(\omega_C)\dual
\to \H^1 \bigl(T_{\P^r} \otimes L\dual \bigr) 
\to 0,
\]
from which it follows that:
\begin{gather}
\label{H^0T_P(-1)}
0 \to \H^0 (L)\dual \to 
\H^0 \bigl( \restr {T_{\P^r}} C \otimes L\dual \bigr)
\to \ker \bigl(\trsp\mu_{L,\omega_C}\bigr) \to 0
\qquad \text{is an exact sequence;} \\
\H^1 \bigl( \restr {T_{\P^r}} C \otimes L\dual \bigr)
\cong
\coker \bigl(\trsp\mu_{L,\omega_C}\bigr)
\cong R(L,\omega_C)\dual.
\end{gather}
Then, $\trsp \Phi_{L,\omega_C}$ identifies with the map
\begin{equation}
\H^1 \bigl( T_C \otimes L\dual \bigr) 
\to \H^1 \bigl( \restr {T_{\P^r}} C \otimes L\dual \bigr)
\end{equation}
induced by the inclusion $T_C \subset \restr {T_{\P^r}} C$.
Eventually, if $C \subset \P^r$ is neither a line nor a conic, 
the normal bundle exact sequence twisted by $L\dual$ gives the exact
sequence
\begin{equation}
\label{exsq-H^0N(-1)-gnl}
0 \to 
\H^0 \bigl( \restr {T_{\P^r}} C \otimes L\dual \bigr)
\to \H^0 \bigl( N_{C/\P^r} \otimes L\dual \bigr)
\to \ker (\trsp \Phi_{L,\omega_C})
\to 0.
\end{equation}
When $C$ has positive genus, the map $\mu_{L,\omega_C}$ is surjective
\cite {arbsern,ciliberto83, green84},
and Lemma~\ref{l:wahl} follows from \eqref{H^0T_P(-1)} and
\eqref{exsq-H^0N(-1)-gnl}.

\paragraph{}
\label{p:inf-autom}
%Consider the situation of Lemma~\ref{l:wahl}.
The space
$\H^0 (N _{C/\P^{r}} \otimes L^ \vee )$
 is the Zariski tangent space to the space of
deformations of $C$ in $\P^r$ fixing a given hyperplane section
$H\cap C$.
The inclusion of $\H^0 (L) ^\vee$ in this space, 
at the left-hand-side of \eqref{w-exctsq}, 
comes from the isomorphism
$\H^0 (L) ^\vee
\cong \H^0(T_{\P^r} \otimes L^\vee)$
given by \eqref{H^0T_P(-1)},
which identifies $\H^0 (L) ^\vee$ with a space of
infinitesimal automorphisms inside 
$\H^0 (N _{C/\P^{r}} \otimes L^ \vee )$.

It is useful in our setup
to express this identification in coordinates.
Fix homogeneous coordinates $(x_0:\ldots:x_r)$ such that $H$ has
equation $x_0=0$. The space $\H^0(T_{\P^r})$ is the tangent space at
the origin to $\PGL_{r+1}$, and
$H^ 0(T_{\P^r} \otimes L^ \vee)$ is the tangent space at
the origin to the subgroup
\[
\C^r \rtimes \C^* < \PGL_{r+1}
\]
of projectivities fixing $H$ point by point.
The latter are given in the affine chart $x_0=1$ by
$x_i\mapsto \lambda x_i+a_i$ ($1 \leq i \leq r$),
with $(a_1,\ldots,a_r) \in \C^r$ and $\lambda \in \C^*$. 
The elements
of $H^ 0(T_{\P^ r}\otimes L^ \vee)$ thus identify with the
infinitesimal automorphisms given in affine coordinates by
$x_i\mapsto (1+t\epsilon_0) x_i+t\epsilon_i$ ($1 \leq i \leq r$), where $t^
2=0$ and $\boldsymbol\epsilon = 
(\epsilon_0,\epsilon_1,\ldots,\epsilon_r) \in \C^{r+1}$.
The isomorphism $\H^0 (L) ^\vee \cong \H^0(T_{\P^r} \otimes L^\vee)$
maps the linear form
$\mathbf{x} \mapsto \boldsymbol \epsilon \cdot \mathbf{x}$ to the latter
infinitesimal automorphism.

\bigskip
The following statement is a higher dimensional version of
Lemma~\ref{l:wahl}, and the proofs of the two go along the same lines.

\begin{lemma}
\label{l:wahl1} Let $X$ be a smooth 
% irreducible projective 
variety of dimension $n\geq 2$ with $h^ 1(\O_X)=0$, and $L$ a very
ample line bundle on $X$, with $(X,L)$ different from $(\P^n, \O_{\P^
n}(1))$.
If $X$ is a surface, we assume in addition that the multiplication map
$\mu_{L,\omega_X}$
% \H^ 0(L)\otimes \H^ 0(\omega_X)\to 
% \H^ 0(L\otimes \omega_X)$ 
% %(see \eqref{eq:mult-map}) 
is surjective.
We consider $X\subset
\P^ r:=\P(\H^0 (L)^ \vee)$. Then one has the exact sequence
\begin{equation}
\label{w-exctsq2}
0 \to 
\H^0 (L) ^\vee 
\to \H^0 \bigl(N _{X/\P^{r}} \otimes L^ \vee \bigr)
\to \H^ 1(T_X\otimes L^ \vee)
\to 0.
\end{equation}
If $n=2$ and $X$ is a K3 surface, then $\mu_{L,\omega_X}$ is surjective,
and moreover
\[
\H^1 \bigl(N _{X/\P^{r}} \otimes L^ \vee \bigr)\cong \H^ 2(T_X\otimes
L^ \vee).
\]
\end{lemma}
\vspace{\negvcorr}
\vspace{\negvcorr}

\begin{proof}
The Euler exact sequence twisted by $L^\vee$,
% is
% \begin{equation}\label{eq:euler1}
% 0\to L^\vee \to H^ 0(L)^ \vee\otimes \O_X\to  \left.T_{\P^ r} \right|_X\otimes L^ \vee\to 0. 
% \end{equation}
% By the Kodaira vanishing theorem, one has $h^i(L^ \vee)=0$ for
% $i<n$. This implies
together with the the Kodaira Vanishing Theorem imply that
\begin{equation}
\label{eq:h0}
\H^ 0(\left.T_{\P^r} \right|_X\otimes L^ \vee)\cong \H^0 (L) ^\vee.
\end{equation}
Moreover 
\begin{equation}
\label{eq:h1}
\H^ 1(\left.T_{\P^r} \right|_X\otimes L^ \vee)=0:
\end{equation}
this follows from the vanishing of $\H^ 1(\O_X)$ and, 
if $n\geq 3$ the vanishing of $\H^2 (L\dual)$, 
or if $n=2$ the surjectivity of $\mu_{L,\omega_X}$.

Next we consider the twisted normal bundle exact sequence 
\begin{equation}\label{eq:normal-1}
0\to T_X\otimes L^ \vee 
\to \restr {T_{\P^ r}} X \otimes L^ \vee 
\to N_{X/\P^ r}\otimes L^ \vee\to 0.
\end{equation}
By the Mori--Sumihiro--Wahl Theorem \cite {MS78, wahl83}, one has 
$\H^0(T_X\otimes L^ \vee)=0$. Then \eqref {w-exctsq2} follows from 
\eqref{eq:h0}, \eqref {eq:h1}, and the long exact sequence of cohomology
of \eqref {eq:normal-1}.

If $\omega_X$ is trivial then
$\mu_{L,\omega_X}$ is an isomorphism hence, if $n=2$,
$\H^2(\left.T_{\P^r} \right|_X\otimes L^ \vee)$ vanishes, and the 
final assertion follows from the cohomology sequence of \eqref
{eq:normal-1}.
\end{proof}

% \bigskip
% The following lemma is not new , but we recall its proof for
% completeness. 

\begin{lemma}[(see, \eg{} {\cite[\S~2]{wahl87}} and
{\cite[Lemma~2.7 (ii)]{knutsen-lopez}})]
\label{l:fighissimo} Let $X\subset \P^ n$ be a local
complete intersection variety such that the homogeneous ideal of $X$
is generated by quadrics and the first syzygy module is generated by
linear syzygies. Then $\H^ 0(N_X(-k))=0$ for all $k\geq 2$.
\end{lemma}

\bigskip
Note that this applies to any canonical curve $C$ with
$\Cliff(C) >2$ by
\cite{voisin88,schreyer91},
resp.\ to any $K3$ surface $S \subset \P^g$ with
$\Cliff(S,\O_S(1))>2$ \cite{saint-donat}.
In the latter case, Andreas Knutsen
kindly indicated to us how to prove that
$\H^0(N_S(-2))=0$
if $g \geq 11$ without any assumption on the Clifford index.
% the following stronger result:
% let  be any smooth (non-hyperelliptic) $K3$ surface of
% genus $g$. If $g \geq 11$, then $\H^0(N_S(-2))=0$; moreover, for $g
% \leq 10$ there is a complete classification of those $S$ for which 
% $\H^0(N_S(-2)) \neq 0$. 
We don't dwell on this here.

% \begin{proof} The   ideal sheaf of $X$ has a minimal presentation 
% \begin{equation*}\label{eq:pres0}
%  \O_{\P^ {n}}(-3)^ {\oplus
%   m_1} \stackrel {\mathbf r} \longrightarrow   \O_{\P^ {n}}(-2)^ {\oplus m}\longrightarrow \mathcal I_{X/\P^ {n}}\longrightarrow 0
% \end{equation*}
% where ${\mathbf r}$ is an  $m\times m_1$ matrix of linear forms. 
% By dualising and tensoring with $\O(-k)$ we have the exact sequence 
% \begin{equation}\label{eq:fig}
% 0 \longrightarrow N_{X/\P^ n} (-k) \longrightarrow   \O_X(2-k)^
% {\oplus m}  \stackrel {\trsp {\mathbf r}} 
% \longrightarrow \O_X(3-k)^ {\oplus m_1}.
% \end{equation}
% The assertion is straightforward for $k\geq 3$.  The rows of $\mathbf
% r$ are linearly independent over $\C$, since otherwise there would be
% some quadric in the ideal of $X$ not involved in any linear syzygy
% but involved in the quadratic Koszul syzygies,
% contradicting the hypothesis.  This gives the case $k=2$.
% \end{proof}

\section{Ribbons and extensions}
\label{S:ribbons}

In this Section we recall the required background on ribbons, and
their relation with Wahl maps in the case of canonical curves. We
review \cite[Proof of Thm~7.1]{wahl97} in some details, as we will
need this later. We make our observation that
unicity holds in Theorem~\ref{t:w+abs} (Remark~\ref{r:unicita}).

% \subsection{Generalities on ribbons}
% \label{s:ribb}

\paragraph{}
Let $Y$ be a reduced connected scheme and $L$ an invertible sheaf on
$Y$. A \emph{ribbon} over $Y$ with \emph{normal bundle} $L$ (or
\emph{conormal bundle} $L^\vee$) is a scheme $\tilde Y$ such that
$\tilde Y_{\mathrm{red}}=Y$, $\mathcal I_{Y/\tilde Y}^2=0$ and
$L^\vee\cong   \mathcal I_{Y/\tilde Y} 
= \ker\bigl( \O_{\tilde Y} \to \O_Y \bigr)$.
To each ribbon one associates the extension class $e_{\tilde Y}\in
\Ext^1_Y(\Omega^1_Y,L^\vee)$ determined by the conormal sequence of
$Y\subset \tilde Y$:
\[
\xymatrix@=10pt{
& 0 \ar[r] 
& L^\vee\ar[r]\ar@{=}[d]
& \O_{\tilde Y} \ar[r]\ar[d]
& \O_Y \ar[r]\ar[d]
& 0\\
e_{\tilde Y}:
& 0 \ar[r]
& L^\vee\ar[r]
& \Omega^1_{\tilde Y} \bigr| _Y \ar[r]
& \Omega^1_Y\ar[r] 
& 0
}
\]
Note that the upper row is an extension of sheaves of algebras, while
the lower one is an extension of $\O_Y$-modules; the middle and
right vertical arrows are differentials and therefore are not
$\O_{\tilde Y}$-linear.  Conversely, to
each element of $\Ext^1_Y(\Omega^1_Y,L^\vee)$ there is associated a
unique ribbon constructed in a standard way (see, e.g.,
\cite[Thm.~1.1.10]{ilsernesi}).

\subparagraph{}\label{sp:proportional}\medskip
 The trivial extension corresponds to the \emph{split ribbon}, the
 unique one such that the inclusion $Y \subset \tilde Y$ admits a
 retraction   $\tilde Y \to Y$.  Two extensions define isomorphic ribbons
 if and only if they are proportional.  Therefore the set 
of isomorphism classes of non-split
 ribbons is in $1:1$ correspondence with $\P \bigl(
 \Ext^1_Y(\Omega^1_Y,L^\vee) \bigr)$.

\paragraph{}
\label{ex:doublehyp}
Let $Y\subset X$ be a nonsingular hypersurface in a variety $X$ smooth along $Y$. % Consider the conormal sequence of $Y\subset X$:
% \[
% \kappa_{Y/X}:\xymatrix@=10pt{
% 0\ar[r]& N^\vee_{Y/X}\ar[r]& \Omega^1_{X} \bigr|_Y \ar[r]& \Omega^1_Y \ar[r]&0
% }
% \]
% Then
The conormal sequence of $Y\subset X$ yields an element
$\kappa_{Y/X} \in \Ext^1_Y(\Omega^1_Y,N^\vee_{Y/X})$, defining a
ribbon $\tilde Y$ over $Y$ with normal bundle $N_{Y/X}$. A priori we have
another ribbon $\bar{Y}$ over $Y$, defined by $\O_{\bar{Y}}
= \O_X/\mathcal I_Y^2 = \O_X/\O_X(-2Y)$;  by definition, one has
$\bar{Y}\subset X$. On the other hand it follows from
the conormal sequence of $\bar{Y}\subset X$ that
$\Omega^1_{X} \bigr|_{Y} = 
\Omega^1_{\bar{Y}} \bigr|_Y$.  Therefore $\tilde
Y=\bar{Y}$.  We call $\tilde Y$ a \emph{double hypersurface} in $X$
and we denote it by $2Y_X$.

\subparagraph{}\medskip
If $H \subset \P^{r+1}$ is a hyperplane, then $2H_{\P^{r+1}}$ is a
split ribbon. This can be seen in two ways. Firstly, projecting from a
point $p\notin H$ we obtain a retraction $2H_{\P^{r+1}} \to
H$. Alternatively, the extension defining $2H_{\P^{r+1}}$ belongs to
$\Ext^1_H(\omega_H,\O_H(-1)) = \H^1(H,\O_H(r))=0$,
and therefore splits.

% \subsection{Gaussian maps and ribbons}
% \label{s:gauss-ribbs}

\paragraph{}
\label{p:setup-lemmaribb}
Consider a smooth variety $X \subset \P^n$, and identify this
$\P^n$ with a hyperplane $H \subset \P^{n+1}$. 
Let $L= \O_X(1)$.
% We have the conormal exact sequence of $X \subset H$
% \[
% 0\to N^\vee_{X/H} \to \restr {\Omega^1_{H}} X
% \xrightarrow { r} \Omega^1_X\to 0,
% \]
% and 
The restriction map 
$r:\restr {\Omega^1_{H}} X \to \Omega^1_X$ induces a map
\[
\eta: \Ext^1_X(\Omega^1_X, L\dual)
\to 
\Ext^1_X(\restr {\Omega^1_H} X, L\dual).
\]
The following result characterizes in terms of this map $\eta$ those
abstract ribbons $\tilde X$ over $X$ with normal bundle $L$, which can
be embedded in the embedded ribbon $2H_{\P^{n+1}} \subset \P^{n+1}$
in a way compatible with the embedding $X \subset \P^n=H$.

\begin{lemma}[(see {\cite[\S~0]{voisin-acta}})]
\label{l:ribb-obstr}
In the situation of \ref{p:setup-lemmaribb}, consider
an element $e\in \Ext^1_X(\Omega^1_X, L\dual)$, and let $\tilde X$ be
the ribbon over $X$ defined by $e$.
There exists an inclusion $\tilde X \subset 2H_{\P^{n+1}}$ such that
$X = \tilde X \cap H$
if and only if $\eta (e)=0$.
\end{lemma}

\begin{proof}
Consider the following diagram:
\[
\xymatrix@=10pt{
\restr {e_{2H_{\P^{n+1}}}} X :&
0\ar[r] & L^\vee \ar@{=}[d] \ar[r] 
& \restr {\Omega^1_{\P^{n+1}}} X \ar@{-->}[d] \ar[r]
& \restr {\Omega^1_{H}} X \ar[r] \ar@{=}[d]
& 0 \\
\eta(e):&
0 \ar[r] & L^\vee \ar@{=}[d] \ar[r] & \mathcal E \ar[d]\ar[r]
& \restr {\Omega^1_{H}} X \ar[r] \ar[d]^-{r} 
& 0 \\
e: & 0\ar[r]& L^\vee\ar[r] 
& \restr {\Omega^1_{\tilde X}} X  \ar[r]& \Omega^1_X \ar[r]&0
}\]
There exists an inclusion $\tilde X \subset 2H_{\P^{n+1}}$ such that
$X = \tilde X \cap H$
if and only if there exists a dashed arrow such that the diagram
commutes,
and the latter condition is equivalent to
$\eta(e) \cong \restr {e_{2H _{\P^{n+1}}}} X$.
The result follows by the fact that the ribbon $2H_{\P^{n+1}}$ is
split, whence $e_{2H _{\P^{n+1}}}=0$.
\end{proof}

\bigskip
In particular, Lemma~\ref{l:ribb-obstr} tells us that if $\eta$ is
injective, then every ribbon $\tilde X \subset 2H_{\P^{n+1}}$ such
that $X = \tilde X \cap H$ is split.

\begin{lemma}
\label{l:id-eta}
When $X=C$ is a canonical curve (resp.\ $X=S$ is a $K3$ surface),
the restriction to $\bigwedge^2 \H^0(\omega_C)$ of the map
$\eta$ is $\trsp \Phi_C$ (resp.\ 
the map $\eta$ is $0$).
\end{lemma}

\begin{proof}
In the case of a canonical curve, the stated identification 
% of the map $\eta$ with $\trsp \Phi_C$ 
is merely the definition of the Wahl map $\Phi_C$,
see %subsection~\ref{s:gaussian} and 
\ref{p:identifications}.
In the case of a $K3$ surface, the target of the map $\eta$ is 
$\H^1( \restr {T_{\P^g}} S (-1))$, which vanishes by \eqref{eq:h1}.
\end{proof}

\paragraph{}
\label{p:unobstr-ribbons}
Let $C \subset \P^{g-1}$ be a canonical curve.
The upshot of the previous paragraphs is that
the  space $\P (\ker (\trsp\Phi_C))$ canonically
identifies with the space of isomorphism classes of ribbons $\tilde C$
over $C$ for which there may be a surface $S \subset \P^g$ not a cone,
such that $\tilde C = 2C_S$.

Similarly, for a $K3$ surface $S \subset \P^g$,
the space $\P(\H^1(T_S(-1)))$ parametrizes isomorphism
classes of ribbons $\tilde S$ over $S$ that may come from a threefold
$V \subset \P^{g+1}$ not a cone, having $S$ as a hyperplane section.
The vanishing of $\eta$ for $K3$ surfaces tells us that,
quite surprisingly, if a ribbon over $S$ has the appropriate normal
bundle $\O_S(1)$, there is no obstruction to embed it as an
infinitesimal threefold in $\P^{g+1}$ with hyperplane section $S$.

% \subsection{Wahl's extension construction}
% \label{s:wahl-ext}

% Here we recall some detail of Wahl's extension construction 
% \cite[Proof of Thm~7.1]{wahl97};
% it follows Stevens' approach, see, \eg \cite{stevens}.
% These are needed both for our
% construction of the universal extension in
% subsection~\ref{s:extension}, and to justify our observation
% (in the next subsection~\ref{s:unicita}) that
% under appropriate assumptions the extension corresponding to a given
% ribbon is unique (Remark~\ref{r:unicita}).

\paragraph{}
\label{p:wahl-ext-cstr}
Let us now recall some detail of Wahl's extension construction 
\cite[Proof of Thm~7.1]{wahl97};
it follows Stevens' approach, see, \eg \cite{stevens}.
Let $C\subset \P^ {g-1}$ be a canonical curve with $\Cliff(C) >2$,
hence of genus $g\geqslant 7$. Let $\mathbf x=(x_0:\ldots:x_{g-1})$ be
homogeneous coordinates in $\P^ {g-1}$ and let $\mathbf f(\mathbf
x)=\mathbf 0$ be the homogeneous quadratic equations of $C$ in the
form of a vector of length $m$.
Since $\Cliff(C) >2$, we know by \cite{voisin88,schreyer91} that the
homogeneous ideal of $C$ has a minimal presentation
\begin{equation}
\label{eq:pres} 
\O_{\P^ {g-1}}(-3)^ {\oplus m_1}
\stackrel {\mathbf r} \longrightarrow   
\O_{\P^ {g-1}}(-2)^ {\oplus m}
\stackrel {\mathbf f}\longrightarrow 
\mathcal I_{C/\P^ {g-1}}\longrightarrow 0.
\end{equation}

Assume now that $g\geq 11$. By \cite[Theorem~3]{abs1} one has $\H^
1(\P^{g-1},\mathcal I^ 2_{C/\P^ {g-1}}(k))=0$ for all $k\geq 3$, so
that 
\cite[Theorem 7.1]{wahl97} can be applied: Consider a non-zero $v\in
\coker(\Phi_C)^ \vee$, and let $C_v$ be the ribbon over $C$
corresponding to $v$; 
this ribbon $C_v$ lies in $\P^g$ by Lemmas~\ref{l:ribb-obstr} and
\ref{l:id-eta}; the construction in the proof of \cite [Theorem 7.1]
{wahl97} provides a surface $S_v$ in $\P^ g$ such that
$C_v = 2C_{S_v}$.

We shall now outline this construction.
Because of \eqref {w-exctsq},   $ \coker(\Phi_C)^ \vee$
is a quotient of $\H^ 0(C, N _{C/\P^{g-1}}(-1))$;
we choose a lift of $v$ with respect to this quotient.
The inclusion of
$\H^ 0(C, N _{C/\P^{g-1}}(-1))$ in $\H^ 0(C,\O_C(1))^ {\oplus m}$
coming from \eqref {eq:pres} represents this lift of $v$ as a length
$m$ vector $\mathbf f_v$ of linear forms on $\P^{g-1}$. The scheme
$C_v$ is defined by the equations
\begin{equation}
\label{eq:infdef}
\mathbf f(\mathbf x)+t\mathbf f_v(\mathbf x)=\mathbf 0,\quad  t^
2=0
\end{equation}
in the $g$-dimensional projective space with homogeneous
coordinates $(\bx:t)=(x_0:\ldots:x_{g-1}:t)$. Wahl proves that there is a
vector $\mathbf h_v$ of constants such that $S_v$ is defined by the
equations  
\begin{equation}\label{eq:eqsurf}
\mathbf f(\mathbf x)+t\mathbf f_v(\mathbf x)+ t^ 2 \mathbf h_v=\mathbf 0.
\end{equation}

% \subsection{Unicity of extension}
% \label{s:unicita}

% We now study the unicity of the surface extending a ribbon constructed
% in \ref{p:wahl-ext-cstr} above.
% The main result is Remark~\ref{r:unicita}; the assumptions there
% ensure that \cite[Theorem~3]{abs1} and \cite[Theorem~7.1]{wahl97}
% apply, as in \ref{p:wahl-ext-cstr}.

\begin{remark}
\label{r:unicita}
{\it
Let $C\subset \P^ {g-1}$ be a canonical curve of genus $g \geq 11$
and Clifford index $\Cliff(C) >2$.
Given a ribbon $v \in \ker (\trsp \Phi_C)$ over $C$, there is a
surface $S_v \subset \P^g$ extending it;
it is uniquely determined up
to the action of a group of projective transformations of 
$\P^g$ pointwise fixing $C$, whose tangent space identifies with 
$\H^0(\omega_C) \dual$ (see \ref{p:inf-autom}).
}%
\end{remark}

\bigskip
This is a mere consequence of \cite[Proof of Thm.~7.1]{wahl97}.
In a nutshell, the idea is that any $1$-extension of $C$ is given by
equations as in \eqref{eq:eqsurf}, where, by sequence \eqref
{w-exctsq}, $\mathbf f_v$ is determined by $v$ up to an element of
$\H^0(\omega_C)^\vee$, i.e.\ up to an infinitesimal automorphism 
as in \ref{p:inf-autom},
which does not change the isomorphism class of the ribbon
$C_v$. Then the extension $S_v$ depends
only on the choice of $\mathbf h_v$. 
Now any two such choices differ by an element of 
$\H^0(C,N_{C/\P^{g-1}}(-2))$ as we recall in \ref{p:unic-pf} below,
and this space is zero by Lemma~\ref{l:fighissimo} because $\Cliff (C)
>2$. 

\paragraph{}
\label{p:unic-pf}
To justify our affirmations above, let us briefly recall how the
vector of constants $\mathbf h_v$ may be chosen in
\cite[Proof of Thm.~7.1]{wahl97}.
Set
$S= \Sym ^\bullet \H^0(C,\omega_C)$, $I_C \subset S$ the homogeneous
ideal of $C$ in $\P^{g-1}$, and $S_C = S/I_C$.
In terms of graded $S$-modules, the presentation \eqref{eq:pres}
writes
\begin{equation}
\label{E:presentA}
S(-3) ^{\oplus m_1} 
\stackrel {\mathbf r} \longrightarrow
S(-2) ^{\oplus m} 
\stackrel {\mathbf f} \longrightarrow
S \longrightarrow S_C \longrightarrow 0.
\end{equation}
We need to recall the definition of $T^2_{S_C}$ from \cite[\S~3.1.2]{ilsernesi}. Denote by
\[
{R}_C:=\ker(\mathbf f)= \im(\mathbf r)
\subset S(-2) ^{\oplus m}
\]
the \emph{graded module of relations}.  It contains the graded
submodule ${R}_0$ of \emph{trivial} (or \emph{Koszul})
\emph{relations}.
An elementary
remark shows that ${R}_C/{R}_0$ is killed by $I_C$ and
therefore it is an $S_C$-module. Thus the presentation
\eqref{E:presentA} induces an exact sequence:
\[
{R}_C/{R}_0 \longrightarrow S_C(-2) ^{\oplus m} \longrightarrow
I_C/I_C^2 
\longrightarrow 0.
\]
The following exact sequence,
\begin{equation}\label{E:T2}
\xymatrix@=12pt{
0\to \mathrm{Hom}(I_C/I_C^2,S_C)\ar[r]
&\mathrm{Hom}(S_C^{\oplus m}(-2),S_C)\ar[r]^(.4){\tilde \ber} \ar@{=}[d]
&\mathrm{Hom}({R}_C/{R}_0,S_C)\to
T^2_{S_C} \to 0\\
&\mathrm{Hom}(S_C^{\oplus m},S_C(2))
}%endxymatrix
\end{equation}
defines $T^2_{S_C}$.

By flatness of the family of affine schemes over 
$\Spec (\C[t]/ (t^2))$ defined by \eqref{eq:infdef},
the relations $\ber$ lift, \ie there is
an $m \times m_1$ matrix of constants $\ber _v$ such that
\begin{equation*}
(\bef + t \bef_v)
(\mathbf r + t \mathbf r _v)
=0 \mod t^2,
\quad \text{i.e.} \quad
t(\bef_v \ber + \bef \ber _v)
=0.
\end{equation*}
Now the vector of constants $\bh_v$ is only subject to the condition
that the equations \eqref{eq:eqsurf} define a flat family of affine
schemes over $\Spec (\C[t]/ (t^3))$
(this, as in the proof of \cite[Proof of Thm.~7.1]{wahl97}, eventually
ensures flatness over $\Spec (\C[t])$), which in turn boils down to
\begin{equation}
\label{cond:flat}
t^2 (\bef_v \ber_v+ \bh_v \ber)=0.
\end{equation}
The map $\bef_v \ber_v: S ^{\oplus m_1} \to S(1) \to S_C(1)$ induces a
map belonging to $\Hom (R_C/R_0, S_C) _{-2}$, and 
condition~\eqref{cond:flat} is equivalent to 
\[
\bh_v \in 
\Hom (S^{\oplus m},S)_0 
= \Hom (S_C^{\oplus m},S_C)_0
= \Hom (S_C^{\oplus m},S_C(2))_{-2}
\]
being a lift of $-\bef_v \ber_v \in \Hom (R_C/R_0, S_C) _{-2}$ with
respect to the map $\tilde\ber$ in the exact sequence \eqref{E:T2};
two such lifts differ by an element of
\[
\Hom (I_C/I_C^2,S_C)_{-2} = \H^0(C,N_{C/\P^{g-1}}(-2)) = (0).
\]
\qed

\medskip
%At this point, it would be a shame not to mention the fact that 
The
\emph{existence} of a lift $\bh_v$, on the other hand, comes from the
identification 
$(T^2 _{S_C}) _{-2} \dual \cong \H^1 (\P^{g-1}, \mathcal I _{C/
  \P^{g-1}}(-3))$ 
\cite[Cor.~1.6]{wahl97} and the vanishing of the latter cohomology
group \cite[Thm.~3]{abs1}.

\bigskip
\subparagraph{Observation}
\label{obs:quadratic}
In the above proof, as both $\bef_v$ and $\ber_v$ depend linearly on
$v$, and $\bh_v$ is a lift of $-\bef_v\ber_v$, the vector of constants
$\bh_v$ depends quadratically on $v$.

\begin{remark}
\label{r:unic-BM}
It is not always true that the extension of a ribbon over a canonical
curve is unique.
Beauville and Mérindol \cite[Proposition~3 et
Remarque~4]{beauville-merindol} 
classify the curves $C$ for which there is a $K3$ surface extending
the trivial ribbon over the canonical model of $C$ (this indeed
contradicts the unicity, as in any event the cone over $C$ extends the
trivial ribbon). They show that such a curve is either the
normalization of a plane sextic, or a complete intersection of
bidegree $(2,4)$ in $\P^3$. In both cases one has
$\H^0(C,N_{C/\P^{g-1}}(-2)) \neq 0$; we leave this to the reader.
\end{remark}

% If $C$ is a smooth plane sextic, then its canonical model is a
% hyperplane section of a $K3$ surface complete intersection of the cone
% over the Veronese surface $v_3(\P^2) \subset \P^9$ in $\P^{10}$ and a
% quadric $Q \subset \P^{10}$; the trivial quadratic relations between
% $Q$ and the 
% quadrics defining the cone over $v_3(\P^2)$ give non-trivial elements
% of $\H^0(C,N_{C/\P^{g-1}}(-2))$.
% If $C$ is a $\delta$-nodal plane sextic, its canonical model is
% a hyperplane section of the cone over a Del Pezzo surface of degree
% $9-\delta$ and a quadric, and $\H^0(C,N_{C/\P^{g-1}}(-2)) \neq 0$ as
% before.
% We leave to the reader the analogous analysis when $C$ is a
% non-hyperelliptic plane sextic with arbitrary singularities; in any
% event, such a curve is an equigeneric degeneration of nodal plane
% sextics (see \cite{DS15} and the references therein),
% hence $\H^0(C,N_{C/\P^{g-1}}(-2)) \neq 0$.

% \par Similarly, if $C$ is a smooth complete intersection of
% bidegree $(2,4)$ in $\P^3$, then its canonical model is a hyperplane
% section of a $K3$ surface complete intersection of the cone over the
% Del Pezzo surface given by the anticanonical model of $\P^1 \times
% \P^1$ with a quadric $Q \subset \P^9$. Again, the trivial quadratic
% relations between $Q$ and the 
% equations of the cone produce non--trivial elements
% of $\H^0(C,N_{C/\P^{g-1}}(-2))$.

\paragraph{}
\label{ex:zak-tevelev}
The line of argument of \ref{p:unic-pf} may be applied to the more
general situation in which $C$ is a curve with Clifford index greater
than $2$, embedded by the complete linear system of an arbitrary very
ample line bundle $L$, with the proviso that the multiplication map
$\mu_{L,\omega_C}$ is surjective, which is equivalent to the condition
that $C$ has positive genus.
If the multiplication map is not surjective, then the relation between
$\H^0(N_C\otimes L\dual)$ and 
$\ker( \trsp \Phi_{L,\omega_C})$ is more complicated (see
\ref{p:identifications}), and indeed for rational normal
curves of degree $d>3$, there exist ribbons with several extensions, see
\cite[p.~276]{zak}.
For non-linearly normal curves, the same problem may also appear in
positive genus, \eg for hyperplane sections of irregular scrolls.

\section{Wahl maps and extensions of canonical curves}
\label{S:extension}

This Section is devoted to the proof of our main extension result,
Theorem~\ref{thm:cds2}, and its variant \ref{p:univ-ext}. 
We start by recalling the following auxiliary result.

%\subsection{Two auxiliary results}

% We start by recalling two auxiliary results that will be useful to
% study the properties of the extensions we construct.

\begin{theorem} 
\label{t:cm}
Let $X \subset \P^m$ be a variety of dimension $n$ having
a linear section which is a canonical curve.
Then $X$ is 
arithmetically Gorenstein, normal,
and has canonical sheaf $\omega_X \cong \O_X(2-n)$.
\end{theorem}

\bigskip
This theorem is clear for $n=1$ and follows in general by the
hyperplane principle; proofs in the cases $n=2,3$ may be found in
\cite {epema}, \cite{conte-murre} respectively.

% \begin{theorem} [(Ishii,  \cite{ishii})]
% \label{t:ishii}
% Let $V$ be a $3$-dimensional, normal, Gorenstein projective variety
% such that $-K_V$ is ample.
% If the locus of non-rational singularities of $V$ has dimension $0$,
% then $V$ is a cone over a normal $K3$ surface.
% \end{theorem}

% \subsection{Extension theorem for canonical curves}
% \label{s:extension}

% We now proceed with our extension result and its corollaries.
% For all subsection~\ref{s:extension}, 

\paragraph{}
For the rest of the Section, we let $C \subset \P^ {g-1}$ be
a canonical curve of genus $g\geq 11$ and Clifford index $\Cliff(C)
>2$.
We choose a section
\begin{equation}
\label{eq:sec-coker-normal}
v \in \ker (\trsp \Phi_C)
\longmapsto
\bef_v \in \H^ 0(C, N _{C/\P^{g-1}}(-1))
\end{equation}
of the extension \eqref{w-exctsq} (for $L=\omega_C$) of vector
spaces, and fix homogeneous coordinates 
$(\bx:t) = (x_0:\ldots:x_{g-1}:t)$ on $\P^g$,
so that for all $v \in \ker (\trsp \Phi_C)$ there is a uniquely
determined extension $S_v$ of the ribbon $C_v$ in $\P^g$, given by
equations \eqref{eq:eqsurf}. %(see \S~\ref{S:ribbons}).

\begin{lemma}
\label{lem:homot}
Let $v \in \ker (\trsp \Phi_C)$, $v\neq 0$, and $\lambda\in \C^ *$.
The surface $S_{\lambda v}$ is obtained by applying to $S_v$ the projective
transformation $\omega_{\lambda ^{-1}}: (\mathbf x: t) \mapsto (\mathbf x:
\lambda^{-1} t)$.%
\end{lemma}
%\vspace{\negvcorr}

\begin{proof} 
By linearity of the map \eqref{eq:sec-coker-normal},
the equations of $S_{\lambda v}$ are
\[
\mathbf f(\mathbf x)+\lambda t\mathbf f_v(\mathbf x)+ t^ 2 \mathbf
h_{\lambda v}=\mathbf 0.
\]
Then the equations of the surface $\omega_{\lambda}(S_{\lambda v})$ are 
% By applying the inverse of $\omega_\lambda$ we find the surface with
% equations
\[
\mathbf f(\mathbf x)+t\mathbf f_v(\mathbf x)+ \frac {t^ 2}{\lambda^ 2}
\mathbf h_{\lambda v}=\mathbf 0.
\]
The surface $\omega_{\lambda}(S_{\lambda v})$ thus
contains $C_v$, and therefore coincides with $S_v$.
\end{proof}

\bigskip
We remark that the above proof shows that  $\mathbf h_{v}$ depends
quadratically on $v$, thus giving another justification to our
Observation~\ref{obs:quadratic}. 
%\bluecom{(\ie $\mathbf h_{\lambda v} = \lambda^2 \mathbf h_{v}$)}

\begin{proposition}
\label{prop:vb}
Set $\cork(\Phi_C)= r+1$
and $\P^r = \P(\ker (\trsp \Phi_C))$.
There is a diagram\footnote{%
Beware that here
$\P(\O_{\P^r}^ {\oplus g}\oplus \O_{\P^r}(1))$ denotes the projective
bundle of one-dimensional quotients,
% of 
% $\O_{\P^r}^ {\oplus g}\oplus \O_{\P^r}(1)$,
whereas everywhere else in the text we use the classical notation for
projective spaces.
}
\begin{equation}
\label{eq:diagr} 
\xymatrix@C=5mm{ \mathcal S
\ar@{}[r]|(.25){\subset} \ar[dr]_{p} 
& \P(\O_{\P^r}^ {\oplus g}\oplus \O_{\P^r}(1)) \ar[d]^ \pi
\\ & \P^r }
\end{equation}
where $p: \mathcal S\to \P^r$ is a flat family of surfaces such that:
\begin{inparaenum}\\
\item the intersection of $\mathcal S$ with $\P(\O_{\P^r}^ {\oplus g})\cong \P^
{g-1}\times \P^ r$ is equal to $C\times \P^ r$;\\
\item
\label{it:univ}
for any $\xi=[v]\in \P^ r$, the inclusion  $\mathcal S_\xi\subset
\P(\O_{\P^r}^ {\oplus g}\oplus \O_{\P^r}(1))_\xi\cong \P^ g$ of fibres
of $p$ and $\pi$, is the
extension $S_v$ of $C=\mathcal S\cap \P(\O_{\P^r}^ {\oplus g})_\xi$.
\end{inparaenum}
\end{proposition}

\begin{proof}
For simplicity we will do the case $r=1$, the general case being
similar.  Let $v_0,v_1$ be a basis of $\coker(\Phi_C)^ \vee$. For
$i=0,1$, consider the diagrams
\begin{equation}
\label{eq:diagr0}
\xymatrix@C=5mm{
\mathcal S_i
\ar@{}[r]|(.25){\subset} \ar[dr]_{p_i} 
& \P^g\times \A^1 
\ar[d]^{\pi_i} \\ 
& \A^1
}
\end{equation} 
given by the equations
\[
\mathbf f (\mathbf x) + t\mathbf f_{v_0+a_1 v_1}(\mathbf x)
+ t^ 2 \mathbf h_{v_0+a_1 v_1}=\mathbf 0, 
\quad \text{resp.}\quad 
\mathbf f(\mathbf x)+t\mathbf f_{a_0 v_0+ v_1}(\mathbf x)+ 
t^ 2 \mathbf h_{a_0 v_0+  v_1}=\mathbf 0,
\]
where $a_1$, resp.\ $a_0$, are affine coordinates on $\A^ 1$.  By
Lemma \ref {lem:homot}, these two diagrams are isomorphic over $\A^
1-\{0\}$ via the map 
$([\mathbf x:t],a_1) \in \mathcal S _0 
\mapsto ([\mathbf x:a_1 t], 1/a_1) \in \mathcal S _1$. 
Diagram \eqref {eq:diagr} is obtained by glueing the
diagrams \eqref {eq:diagr0} via this map.
\end{proof}

\begin{corollary}
\label{c:univ-ext}
Under the assumptions of Proposition~\ref{prop:vb}, there is an
arithmetically Gorenstein normal variety $X$ of dimension $r+2$ in
$\P^{g+r}$ with $\omega_X \cong \O_X(-r)$, not a cone, having $C$ as
a linear section, and satisfying the following property: the surface
linear sections of $X$ containing $C$ are in $1:1$ correspondence with
the surface extensions of $C$ in $\P^g$ that are not cones.
\end{corollary}

% According to Definition~\ref{def:universal}, a variety $X$ as in
% Corollary~\ref{c:univ-ext} is called a \emph{universal extension} of
% $C$. 

\begin{proof}
We keep the notation of Proposition~\ref{prop:vb}.
The $\O(1)$ bundle of $\P(\O_{\P^r}^ {\oplus g}\oplus
\O_{\P^r}(1))$ defines a morphism $\phi$ to $\P^ {g+r}$ which is the
blow-up of $\P^{g+r}$ along the $\P^ {g-1}$ image of the trivial
subbundle $\P(\O_{\P^r}^ {\oplus g})$. Let $X=\phi(\mathcal S)$. The
map $\phi_{|\mathcal S}$ is the contraction of $C\times \P^ r
= \mathcal {S} \cap \P(\O_{\P^r}^ {\oplus g})$ 
to $C\subset \P^ {g-1}\subset \P^ {g+r}$.
The fibres of $p$ are isomorphically mapped to the sections
of $X$ with the $\P^ g$'s containing the $\P^{g-1}$. None of these
surfaces is a cone, because the corresponding first order extensions
of $C$ on them are non-trivial. Therefore $X$ is not a cone. 
The property that the surface linear sections of $X$ containing
$C$ are in $1:1$ correspondence with the surface extension of $C$ in
$\P^g$ other than cones follows from assertion~\eqref{it:univ} in
Proposition~\ref{prop:vb} and Remark~\ref{r:unicita}. The rest of the
assertions follows by Theorem \ref {t:cm}.
\end{proof}

\begin{corollary}
\label{C:ext}
Consider a non-negative integer $r$ such that $\cork(\Phi_C) \geq
r+1$.
\subparagraph{}
\label{cor:ext}
There is an arithmetically Gorenstein normal variety $Y\subset \P^
{g+r}$ of dimension $r+2$ with $\omega_Y \cong \O_Y (-r)$, not a
cone, having $C$ as a curve section with $\P^ {g-1}$.
\subparagraph{}
\label{cor:ext2}
Assume $r>0$.
If there is a surface section of $Y$ with at worst ADE
singularities, then the general threefold section $V$ of $Y$ has
canonical singularities.
\end{corollary}

\bigskip
Note that \ref{cor:ext2} applies to a universal extension of $C$ as
soon as $C$ sits on a $K3$ surface with at worst $ADE$ singularities.

\begin{proof}
Assertion~\ref{cor:ext} follows directly from the previous
Corollary~\ref{c:univ-ext}: take $Y$ a linear section of $X$
containing $C$ of the appropriate dimension.
To prove \ref{cor:ext2}, we note that by the argument in
\cite[Introduction]{prokhorov}, if $V$ had non-canonical
singularities, it would be a cone, a contradiction.
\end{proof}

\section{Integration of ribbons over $K3$ surfaces}
\label{S:intK3}

In this Section we prove Theorem~\ref{t:intK3}, to the effect that
any ribbon on a $K3$ surface in $\P^g$ may be integrated to a
unique threefold in $\P^{g+1}$ (under suitable assumptions).
It will be deduced from the integrability of
ribbons on canonical curves by a hyperplane principle. Key to this
principle is the relation between ribbons over a variety in projective
space and over its hyperplane sections, as explained in the following
paragraph.

\paragraph{}
\label{p:hyp-princ}
In this paragraph, we use without further reference the notions and
results recalled in Section~\ref{S:ribbons}.
% subsections~\ref{s:ribb} and \ref{s:gauss-ribbs}.
Let $S$ be a smooth $K3$ surface in $\P^g$, and $C$ a smooth
hyperplane section of $S$. Recall that we denote by $[2C_S] \in 
\P (\ker (\trsp \Phi_C)) \subseteq \P(\H^1(T_C(-1))$ the ribbon over
$C$, considered up to isomorphism,  given by its being a hypersurface in
$S$.

Let $[\tilde S] \in \P(\H^1(T_S(-1)))$.  This is the isomorphism class
of the ribbon
$\tilde S \subset \P^{g+1}$ over $S$, contained in the ribbon 
$2(H_S)_{\P^{g+1}}$ over the
hyperplane $H_S = \vect S$, and such that $\tilde S \cap H_S = S$.

Now for any hyperplane $H \subset \P^{g+1}$ containing $C$, 
$H \neq H_S$, 
the intersection $H \cap \tilde S$ is a ribbon $C_H$ over $C$ in $H
\cong \P^g$, contained in the ribbon $2 \vect C _{H}$ over $\vect C
\cong \P^{g-1}$,
and such that $C_H \cap \vect C =
C$.
As such, it determines a point of $\P (\ker (\trsp \Phi_C))$.

Thus, the pencil of hyperplanes of $\P^{g+1}$ containing $C$ defines a
line in $\P (\ker (\trsp \Phi_C))$ passing through the point $[2C_S]$;
in other words, $\tilde S$ defines a 
point $\mathfrak l_{\tilde S} \in 
\P \bigl(\ker (\trsp \Phi_C)\bigr) / [2C_S]$.\footnote
{Here and in the rest of this Section, we use the following
non-standard but convenient notation: if $W$ is a vector subspace of
$V$, we write $\P(V)/\P(W)$ for $\P(V/W)$.
}
We are abusing terminology here, as this ``line'' may actually be
reduced to the sole point $[2C_S]$ if all $H$ containing $C$ cut out the
same ribbon over $C$ on $\tilde S$, and in this case $\mathfrak
l_{\tilde S}$ is not well-defined; it will be a consequence of
\ref{p:K3integr} below that this does not happen.

\refstepcounter{paragraph}
\begin{proof}[{\bf \theparagraph} Proof of the existence part
of Theorem~\ref{t:intK3}]
\label{p:K3integr}
We identify $S$ with its
image in $\P^g = |L|\dual$.
Choose any smooth hyperplane section $C$ of $S$. It satisfies the
same assumptions as $S$ on the genus and Clifford index, so we may
consider its universal extension $X \subset \P^{g+r}$ constructed in
Corollary~\ref{c:univ-ext},
% (see \ref{p:univ-ext} and Section~\ref{S:extension} for the notion of
% universal extension)
with
\[
r = \cork (\Phi_C) - 1 = h^1(T_S(-1)),
\]
the second equality in this equation coming from
Corollary~\ref{c:intrsting-rmk}.
By Corollary~\ref{c:univ-ext}, we may consider $S$ as a linear section
of $X$.

Now, every linear $(g+1)$-subspace $\Lambda$ of $\P^{g+r}$  containing
$S$ cuts out a threefold $X_\Lambda$ on $X$ having $S$ as a hyperplane
section, 
hence determines a ribbon 
$2S_\Lambda := 2S_{X \cap \Lambda} \in \H^1(T_S(-1))$,
which in turn determines a point of 
$\P \bigl(\ker (\trsp \Phi_C)\bigr) / [2C_S]$
via the mechanism described in \ref{p:hyp-princ}.
We thus have a composed map
\begin{equation}
\label{eq:ls-rib}
\psi_S: \Lambda \in \P^{g+r} / \vect S %\cong \P^{r-1}
\longmapsto [2S_\Lambda] \in \P (\H^1 (T_S(-1)))
\longmapsto 
\mathfrak l_{2S_\Lambda} \in
\P \bigl(\ker (\trsp \Phi_C)\bigr) / [2C_S], %\cong \P^{r-1}
\end{equation}
albeit maybe only defined so far on a (possibly empty!) Zariski open
subset of $\P^{g+r} / \vect S$ because of the abuse of terminology
mentioned in \ref{p:hyp-princ}.

We claim that the universality of $X$ implies the surjectivity of
$\psi_S$.
Consider a point of 
$\P \bigl(\ker (\trsp \Phi_C)\bigr) / [2C_S]$, and
represent it as a point $[\tilde C]$ of $\P (\ker (\trsp \Phi_C))$
distinct from $[2C_S]$. The universality of $X$ tells us that there
exists a linear $g$-subspace $\Gamma$ of $\P^{g+r}$ such that
$\tilde C = 2C_{X \cap \Gamma}$.
Then $\Lambda:=\vect {\Gamma, S}$ is a $(g+1)$-subspace of $\P^{g+r}$
such that, by construction, $\psi_S(\Lambda) = [\tilde C]$. This
proves our claim.

Now note that in diagram~\eqref{eq:ls-rib}, all three projective spaces
have the same dimension $r-1$, and the two maps whose composition is
$\psi_S$ are linear.
Therefore, the map $\psi_S$ may be surjective only if it is an
isomorphism, and the two maps in \eqref{eq:ls-rib} are isomorphisms as
well. 
We conclude by observing that the surjectivity of the
first map in \eqref{eq:ls-rib}
tells us that for every isomorphism class of ribbons $[\tilde S] \in
\P(\H^1(T_S(-1)))$ there is a threefold $X \cap \Lambda$ such
that $[\tilde S] = [2S_{X \cap \Lambda}]$.
\end{proof}

\refstepcounter{paragraph}
\begin{proof}[{\bf \theparagraph} Proof of the unicity part
of Theorem~\ref{t:intK3}]
\label{p:unicK3}
%Let $S$ be a $K3$ surface as in \ref{p:K3integr}, and 
Consider two
threefold extensions $V$ and $V'$ of $S$ such that the two
corresponding ribbons $2S_V$ and $2S_{V'}$ are proportional.
It follows from the considerations in \ref{p:hyp-princ} and
the unicity of integrals of ribbons over canonical curves
(Remark~\ref{r:unicita}), that the
two threefolds $V$ and $V'$
respectively contain two isomorphic pencils of hyperplane
sections, and this implies that they are isomorphic.
\end{proof}

\begin{remark}
In the case of the trivial ribbon, the conclusion of \ref{p:unicK3} is
that if a $K3$ surface as in 
%\ref{p:K3integr} 
Theorem~\ref{t:intK3}
sits on a threefold $V
\subset \P^{g+1}$, not a cone, then the conormal exact sequence of $S$
in $V$ is not split. By reproducing the argument of
\cite[Proposition~3]{beauville-merindol}, this implies that there does
not exist any automorphism of $V$ of order $2$ and with $S$ as fix
locus.
\end{remark}

\section{Study of the moduli maps}

This Section 
% is devoted to the study of the moduli maps $c_g$ and
% $s_g$ (see Section~\ref{S:prelim}). It 
contains the building
blocks of the proofs of Theorems~\ref{thm:cds} and \ref{thm:cdsK3}.
The following Corollary of Theorems \ref{t:w+abs}
and \ref{t:intK3} comes in a straightforward manner
once one understands the latter Theorems as integration results for
ribbons.

% \subsection{Fibres}

% In this subsection we give two corollaries of theorems \ref{t:w+abs}
% and \ref{t:intK3} respectively, which come in a straightforward manner
% once one understands the latter theorems as integration results for
% ribbons.

\begin{corollary}\label{cor:palese-vero}
Let $(S,C) \in \KC_g$ 
(resp.\ $(V,S) \in \FS_g$)
be such that $\Cliff (C) >2$
(resp.\ $\Cliff (S, -\restr {K_V} S) >2$). Then
\begin{equation*}%\label{eq:estdim}
 \dim(c_g^ {-1}(C))\geq \cork (\Phi_C)-1
\qquad
\text{(resp.\ }
 \dim(s_g^ {-1}(S))\geq h^1\bigl( T_S(-1)
\bigr)-1
\text{).}
\end{equation*}
\vspace{\negvcorr}
\end{corollary}

\begin{proof}
Consider the family 
$p: \mathcal{S} \to \P (\ker (\trsp \Phi_C))$ constructed in
Proposition~\ref{prop:vb}. The $K3$ surface $S$ is a fibre of $p$, so
the fibre $S_{[v]}$ of $p$ over the general $[v] \in \P (\ker (\trsp
\Phi_C))$ 
is a $K3$ surface as well, hence gives rise to a point 
$(S_{[v]},C) \in c_g^{-1}(C)$.
We claim that these points are pairwise distinct, from which the
assertion follows at once.
\par Let $[v],[v']$ be two distinct points of $\P (\ker (\trsp
\Phi_C))$. If $S_{[v]}$ and $S_{[v']}$ are not isomorphic, then the
claim is trivial; else, we may assume $S_{[v]} = S_{[v']}$, and call
this surface $S_0$. There are two copies $C_{[v]}$ and $C_{[v']}$ of
$C$ in $S_0$, and since $[v] \neq [v']$, the respective infinitesimal
neighbourhoods of $C_{[v]}$ and $C_{[v']}$ in $S_0$ are not
isomorphic, which implies that $C_{[v]}$ and $C_{[v']}$ correspond to
two distinct points of the linear system $|\O _{S_0} (C)|$
and there is no automorphism of $S_0$ sending one of the two curves to
the other. This proves the first instance of the statement.

The proof of the second instance is exactly the same, after one notes
that there exists a family
$p: \mathcal V \to \P(H^1( T_S(-1)))$
with properties
analogous to those of the previous family 
$p: \mathcal{S} \to \P (\ker (\trsp \Phi_C))$,
as follows from the arguments in Section~\ref{S:intK3}:
this is Theorem~\ref{thm:cds2K3}!
\end{proof}

% \bigskip
% The following is the analogue of Corollary~\ref{cor:palese-vero} for
% the moduli map $s_g$. Its proof is exactly the same.

% \begin{corollary}
% \label{cor:palese-veroK3}
% Let $(V,S) \in \FS_g$ be such that $\Cliff (S) >2$. Then
% \begin{equation*}%\label{eq:estdim}
%  \dim(s_g^ {-1}(S))\geq h^1\bigl( T_S(-1)
% \bigr)-1.
% \end{equation*}
% \vspace{\negvcorr}
% \end{corollary}

% \begin{proof}
% It follows from the arguments in Section~\ref{S:intK3} that there
% exists a family 
% \[
% p: \mathcal V \to \P(H^1( T_S(-1)))
% \]
% with properties
% analogous to those of the family 
% $p: \mathcal{S} \to \P (\ker (\trsp \Phi_C))$ in the proof of
% Corollary~\ref{cor:palese-vero}: this is Theorem~\ref{thm:cds2K3}!
% Then, one argues exactly as in the proof of 
% Corollary~\ref{cor:palese-vero}.
% \end{proof}

% \subsection{Control of the ramification}

% \paragraph{}
% Let $X$ be a smooth, projective, irreducible variety and $Y$ a codimension 1 subscheme of $X$. 
% The sheaf $T_X\langle Y \rangle$ is by
% definition the kernel of the surjective composed map
% \[
% T_X \to \left. T_X \right|_Y \to N_{Y/X}.
% \]
% It fits in the exact sequence
% \begin{equation}
% \label{eq:tw-restr}
% 0 \to T_X(-Y) \to T_X\langle Y \rangle \to
% T_Y \to 0.
% \end{equation}

\bigskip
The two following results bound from above the dimensions of the
kernels of the differentials of $c_g$ and $s_g$.

\begin{lemma}[(see \cite{ilsernesi}, \S~3.4.4)]
\label{l:sernesi}
Let $(S,C) \in \KC_g$ (resp.\ $(V,S)\in \FS_g$). 
The kernel of the differential of $c_g$ at $(S,C)$ 
(resp.\ of $s_g$ at $(V,S)$) is
% the morphism
% \begin{equation*}
% \H^1 \bigl( T_S\langle C \rangle \bigr) \to
% \H^1(T_C) \quad (\text{resp.} \H^1 \bigl( T_V\langle S \rangle \bigr) \to
% \H^1(T_S))
% \end{equation*}
% coming from the exact sequence \eqref{eq:tw-restr}.
% Its kernel is 
$\H^1\bigl(T_S(-C)\bigr)$ (resp.\ $\H^1\bigl(T_V(-S)\bigr)$).
\end{lemma}

\begin{proposition}\label{prop:ineq}
\label{prop:ineq2}
 Let $(S,C) \in \KC_g$ 
(resp.\ $(V,S) \in \FS_g$)
be such that $\Cliff (C) >2$
(resp.\ $\Cliff (S, -\restr {K_V} S) >2$). Then
\begin{equation*}
% \label{eq:est}
h^1(T_S(-1))+1 \leq \cork (\Phi_C)
\qquad \text{(resp.\ }
h^1(T_V(-1))+1 \leq h^ 1(T_S(-1))
\text{)}.
\end{equation*}
\end{proposition}
\vspace{\negvcorr}

\begin{proof} 
We prove only the first instance of the statement, the other one being
entirely similar.
The curve $C$ is the complete intersection of $S\subset
\P^ g$ with a hyperplane $H\cong \P^ {g-1}$, so one has
\begin{equation}\label{2eq:euler1}
N_{C/\P^ {g-1}}\cong \left. N_{S/\P^ g} \right|_C.
\end{equation}
By Lemma~\ref {l:fighissimo} one has 
$\H^0(N_{S/\P^ g}(-2))=0$, 
so one deduces from the twisted restriction exact sequence
% \begin{equation}\label{2eq:ex}
% 0\to  N_{S/\P^ g} (-2)\to  N_{S/\P^ g} (-1)\to \left. N_{S/\P^ g} \right|_C(-1)\to 0.
% \end{equation}
% and it follows 
that
\begin{equation}\label{2eq:norm2}
h^ 0(N_{S/\P^ g}(-1)) \leq
h^ 0(\left.N_{S/\P^ g} \right|_C(-1)).
\end{equation}
% Consider the twisted restriction exact sequence
% \begin{equation}\label{2eq:ex}
% 0\to  N_{S/\P^ g} (-2)\to  N_{S/\P^ g} (-1)\to \left. N_{S/\P^ g} \right|_C(-1)\to 0.
% \end{equation}
% By Lemma~\ref {l:fighissimo} one has 
% $H^0(N_{S/\P^ g}(-2))=0$, 
% and it follows that
% \begin{equation}\label{2eq:norm2}
% h^ 0(N_{S/\P^ g}(-1)) \leq
% h^ 0(\left.N_{S/\P^ g} \right|_C(-1)).
% \end{equation}
We may now conclude:
\begin{alignat*}{2}
\cork (\Phi_C)+g &= h^0(N_{C/\P^{g-1}}(-1)) 
&\qquad& \text{by Lemma~\ref{l:wahl}} \\
% &= h^0(N_{C/\P^{g}}(-1)) -1 
% && \text{by \eqref{2eq:split}} \\
&= h^0( \left. N_{S/\P^ g}(-1) \right|_C)
&& \text{by \eqref{2eq:euler1}} \\
&\geq h^0(N_{S/\P^ g}(-1))
&& \text{by \eqref{2eq:norm2}} \\
&= h^1(T_S(-1)) + g+1
&& \text{by Lemma~\ref{l:wahl1}.} 
\end{alignat*}\end{proof}

% \bigskip
% Similar to the curve case, the following proposition bounds the
% dimension of the kernel of the differential of $s_g$. Its proof is
% mutatis mutandis the same as that of Proposition~\ref{prop:ineq}
% above.

% \begin{proposition} \label{prop:ineq2} Let $(V,S) \in \FS_g$, and
% suppose that $\Cliff (S) >2$. Then
% \begin{equation}\label{eq:est2}
% h^1(T_V(-1))+1 \leq h^ 1(T_S(-1)).
% \end{equation}
% \end{proposition}

\section{A general bound on the corank of the Wahl map}
\label{S:bnd-cork}

In this Section we prove that
under our usual assumptions, a given canonical curve can be integrated
to a given $K3$ surface in only finitely many ways,
%(Proposition~\ref{pr:finite});
%see also Corollary~\ref{cor:finite2}.
% which says a curve can be realized as
% only finitely many divisors on a given $K3$ surface.
and use this to bound the corank of the Wahl maps of the curves that
sit on a $K3$ surface. %(Corollary~\ref{c:bnd-alpha}).
We first recall the two following results from \cite{epema}. 

% \subsection{Preliminaries on surfaces with canonical curve 
% section}
% \label{s:epema}
%We first recall the two following results from \cite [p.~iii] {epema}. 

\begin{theorem}[{\cite [p.~iii] {epema}}]
\label{thm:epema} 
Let $S$ be a non-degenerate projective surface in $\P^ g$,
having as a hyperplane section a smooth canonical curve 
$C \subset \P^{g-1}$ of genus $g \geq 3$.
Then only the following cases are possible:\\
\begin{inparaenum}
\item $S$ is a $K3$ surface with canonical singularities;\\ 
\item $S$ is a rational surface with a minimally elliptic
singularity, plus perhaps canonical singularities;\\
\item $S$ is a ruled surface over a curve of genus $q\geq 1$ with only
one singularity of \emph{genus} $q+1$, plus perhaps canonical
singularities;\\
\item $S$ is a ruled surface over a curve of genus $q=1$, with two
{simple elliptic singularities}, plus perhaps canonical
singularities.
\end{inparaenum}
\end{theorem}

\medskip\noindent
Surfaces of type
(ii)--(iv) are fake K3 surfaces; their 
Kodaira dimension  is $-\infty$. 
%
% Let us recall that  if $(S,x)$ is the germ of a complex surface singularity and $f: S'\to S$ is the minimal desingularization:\\
% \begin{inparaenum}
% \item [$\bullet$] $(S,x)$ is called \emph {good} if the scheme theoretical fibre of $f$ over $x$ has normal crossings, its components are all smooth and its dual graph has no cycles;\\
% \item [$\bullet$] the \emph{genus} of $(S,x)$ is, by definition, $\dim_\mathbb C(R^ 1f_*(\mathcal O_{S'}))_x$;\\
% \item [$\bullet$] $(S,x)$ is said to be \emph{minimally elliptic} if it is good and the scheme theoretical fibre of $f$ over $x$ has arithmetic genus 1;\\
% \item [$\bullet$]  $(S,x)$ is said to be \emph{simple elliptic} if it is minimal elliptic and the scheme theoretical fibre of $f$ over $x$ is irreducible.
% \end{inparaenum}
%

\begin{proposition}[{\cite [Theorem 2.1, p. 38] {epema}}]
\label{prop:section}
Assume we are in one of the cases (iii)--(iv) of
Theorem~\ref{thm:epema}. 
Let $\mu: S\to \Sigma$ be a minimal model of $S$;
it has a structure of $\P^1$-bundle $f: \Sigma\to D$, where $D$ is a
smooth curve of genus $q$. If the image of $C$ in $\Sigma$ is a
section of $f$, then $S$ is a cone over $C$. 
\end{proposition}

% \subsection{Finiteness of the modular map}

\paragraph{}
\label{p:setup-finite}
Let $C\subset \P^ {g-1}$ be a smooth canonical curve of genus
$g\geqslant 11$ with $\Cliff(C) >2$, 
set $r+1=\cork (\Phi_{C})$ and $\P^r=\P(\coker(\Phi_{C})^ \vee)$, and
assume $r \geq 0$.

Consider the flat family $p:\mathcal S\to \P^r$ constructed in
Proposition~\ref{prop:vb}. Note that no surface of this 
family is a cone over $C$. Suppose that the general member of this
family is a $K3$ surface, possibly with $ADE$ singularities. 
% and let $\Kcan_g$ be the moduli space of
% polarized $K3$ surfaces with $ADE$ singularities (see, \eg
% [Huybrechts, 5.1.4 p.~83]). 
Then we have the rational modular map
\[
s: \P^r \dasharrow \Kcan_g,
\]
whose indeterminacy locus $Z$ consists of the points $[v]\in \P^ r$
such that the corresponding extension $S_v$ of $C$ is a fake $K3$
surface.
So $s$ is defined on the dense Zariski open subset $U=\P^
r-Z$.

\begin{proposition}\label{pr:finite}
The morphism $\restr s U: U\to \Kcan_g$ has finite fibres.
\end{proposition}

\begin{proof} 
We argue by contradiction, and suppose there is an
irreducible curve $\gamma\subset U$ such that $s(\gamma)$ is a
point.
Let us first rule out the possibility that $s$ be defined eveywhere on
$\P^r=\P(\coker(\Phi_{C})^ \vee)$, \ie $Z = \emptyset$.
In this case, since $s$ contracts a curve it must be constant.
Then there would exist a
$K3$ surface $S$ containing a complete $1$-dimensional family of
curves all pairwise isomorphic. Since $\Pic(S)$ is discrete, all these
curves belong to the same linear equivalence class $[C] \in \Pic(S)$.
But since the discriminant locus in $|C|$ is a hypersurface, it is
impossible to have a complete positive-dimensional family of smooth
curves in $|C|$, and we have a contradiction.

We may thus assume that the locus of indeterminacy $Z$ of $s$ is
non-empty. 
Let $\Gamma$ be the Zariski closure of $\gamma$ in $\P^ r$.
The rational map $s$ is defined by a linear system
on $\P^r$ with base locus $Z$. The curve $\Gamma$ necessarily
intersects the divisors in this linear system, and since it is
contracted by $s$ the intersection must be contained in $Z$.
We may thus consider a point $\xi \in \Gamma \cap Z$.
Looking at the normalization of $\Gamma$ at $\xi$, we see
that there is an analytic morphism $ \nu: \D\to \Gamma$, where
$\D$ is the complex unit disc and $\nu(0)=\xi$. By pulling back $p:\mathcal
S\to \P^ r$ to $\D$, we find a flat family $p': \mathcal
S'\to \D$ which is isotrivial over $ \D-\{0\}$, 
with general fibre a $K3$ surface $S_1$,
and central fibre a fake $K3$ surface $S_0$. 
Since the automorphisms of $S_1$ as a polarised surface 
have finite order we may assume, up
to performing a finite base change, that $p'$ is actually trivial over
$\D-\{0\}$.
Moreover there is an inclusion:
\begin{equation*}\xymatrix@C=3mm@R=6mm {
\D\times C 
\ar[dr]_(.4){{\pr_1}} \ar@{}[r]|(.62){\subset}
& \raisebox{.5mm}{$\mathcal S'$} \ar[d]^(.45) {p'}\\
& \D 
}
\end{equation*} 

Now consider a semistable reduction $\tilde p: \tilde S\to
\D$ of $p': \mathcal S'\to \D$; we still have the
inclusion:
\begin{equation}\label{eq:inc}
\xymatrix@C=3mm@R=6mm {
\D\times C 
\ar[dr]_(.4){{\pr_1}} \ar@{}[r]|(.62){\subset}
& \raisebox{.9mm}{$\tilde{\mathcal S}$} \ar[d]^ {\tilde p}\\
& \D 
}
\end{equation} 
The central fibre of $\tilde p$ consists of
the proper transform $\tilde S_0$ of the central fibre $S_0$ of $p$,
plus possibly other components. The central fibre
$C$ of the trivial family $\D\times C \subset \tilde {\mathcal S}$
sits on $\tilde S_0$, and is entirely contained in the
smooth locus of the central fibre of $\tilde p$.

Since $p'$ is trivial over $\D-\{0\}$, so is $\tilde p$. 
This implies that there is a diagram:
\begin{equation*}\xymatrix@R=6mm @C=8mm { 
\raisebox {1.3mm} {$\tilde {\mathcal S}$}
\ar@{-->}[r]^(.4) \psi \ar[dr]_(.4){\tilde p} 
& \D \times S_1 %{\bar {\mathcal S}} 
\ar[d]^(.5) {\pr_1} %{\bar p} \ar@{}[r]|= & \D \times S_1 
\\
& \D 
}
\end{equation*}
where $\psi$ is a birational map contracting all components of the
central fibre of $\tilde p$ but one, and 
$S_1$ is the general fibre of $p'$.
% the central fibre of $\bar p$ is a $K3$ surface isomorphic to the
% general fibre of $p'$, $\tilde p$, and $\bar p$.
By composing the inclusion \eqref {eq:inc} with $\psi$, we still have
an inclusion:
\begin{equation}\label{eq:inc2} 
\xymatrix@C=3mm@R=6mm {
\raisebox {.4mm} {$\D\times C$}
\ar[dr]_(.4){{\pr_1}} \ar@{}[r]|(.5){\subset}
& \D \times S_1 \ar[d]^{\pr_1}
%\raisebox{.9mm}{$\bar {\mathcal S}$} 
%\ar[d]^(.45) {\bar p}
\\
& \D 
}
\end{equation}

We claim that $\tilde S_0$ has to be contracted by $\psi$. Indeed, 
being birational to $S_0$ which is a fake $K3$ surface, $\tilde S_0$
has Kodaira dimension $-\infty$, whereas $S_1$ is a genuine $K3$
surface. 
On the other hand, because of the inclusion \eqref {eq:inc2}, $\tilde
S_0$ has to be contracted to a curve isomorphic to $C$. This implies
that $\tilde S_0$ is a ruled surface over $C$, and so is
$S_0$. Consider a minimal model $\mu: S_0\to \Sigma$ of $S_0$ (and 
of $\tilde S_0$ as well). Then $\Sigma$ is a $\P^1$-bundle $f:
\Sigma\to C$, and the image of $C$ to $\Sigma$ via $\mu$ is a section
of $f: \Sigma\to C$. By Proposition \ref{prop:section}, $S_0$
must be a cone over $C$, a contradiction.
\end{proof}

% \subsection{Applications} %to the corank of the Wahl map}
% \label{s:bnd-cork}

\begin{corollary}
\label{c:bnd-alpha}
Let $C$ be a canonical curve of genus $g\geq 11$ in $\P ^{g-1}$, with
Clifford index $\Cliff(C) >2$. 
If $C$ is a hyperplane section of a $K3$ surface $S$
(possibly with $ADE$ singularities) in $\P^g$,
then %the following inequality holds:
$\cork (\Phi_C) \leq 20$.
\end{corollary}

%\vspace{\negvcorr}

% The key to the inequality of Proposition~\ref{p:bnd-alpha} is the fact
% that a curve $C$ as in the Proposition may only appear finitely many
% times on a given $K3$ surface (this is our Remark~\ref{r:maxmod}).

\begin{proof}
By Proposition~\ref{pr:finite}, there is a rational map
$s: \P(\ker (\trsp \Phi_C)) \dasharrow \Kcan_g$
%(see paragraph~\ref{p:setup-finite}),
which is generically finite on its image. 
Therefore
$\cork (\Phi_C) -1 
\leq \dim (\Kcan_g) = 19$.
\end{proof}

\begin{corollary}[(Proposition~\ref{pr:modmap-finite})]
\label{cor:finite2}
Let $(S,C)\in \KCcan_g$ with $g\geq 11$ and $\Cliff(C)>2$.
There are only finitely many members $C'$ of $|\O_S(C)|$ that are
isomorphic to $C$.
\end{corollary}

\begin{proof}
Assume by contradiction that there is an infinite family $(C_i)$ of
curves isomorphic to $C$ in $|\O_S(C)|$.
By Proposition~\ref{pr:finite}, we may furthermore assume that the
curves $C_i$ all have the same ribbon in $S$.
We consider the pair $(S,C)$ embedded in $\P^g$. Taking $C$ as a
common canonical model for all the curves $C_i$, we obtain a
family $(S_i)$ of surfaces in $\P^g$, such that each $S_i$ is the
image of $S$ by a projectivity of $\P^g$ fixing $C$.
By unicity of the integration of ribbons, see Remark~\ref{r:unicita},
we must have $S_i=S$ for all $i$, and it follows that $S$ has
infinitely many projective automorphisms, a contradiction.
\end{proof}

\begin{remark}
\label{r:corr-sbagl}
It is claimed in \cite[Proposition~1.2 and Corrigendum]{ck14} that for
a smooth curve $C$ sitting on a $K3$ surface $S$, one has
$\H^0(C,\restr {T_S} C)=0$; this would imply the injectivity of the
coboundary map
\[
\partial: \H^0(C,N_{C/S}) \to \H^1(C,T_C)
\]
induced by the conormal exact sequence of $C$ in $S$.
\par
Let $|\O_S(C)| ^\circ$ be the Zariski open subset of $|\O_S(C)|$
parametrizing smooth members of the linear system, and 
$c: |\O_S(C)| ^\circ \to \M_g$ be the morphism mapping
a smooth member $C'$ of
$|\O_S(C)|$ to its modulus in $\M_g$.
Since $\partial$ is the differential of $c$, the injectivity of
$\partial$ would be a stronger result than 
Corollary~\ref{cor:finite2}.
\par
However, there exist smooth curves $C$ sitting on $K3$ surfaces $S$
for which the conormal exact sequence is split \cite[Proposition~3 et
Remarque~4]{beauville-merindol}, see also Remark~\ref{r:unic-BM}.
For such pairs $(C,S)$, the boundary map $\partial$ is downright
zero. This shows that there is a problem with the claim of
\cite[Proposition~1.2 and Corrigendum]{ck14}.
This problem does not affect the results of [ibid.]. 
\par
Note that for a pair $(C,S)$ with split conormal sequence as above,
the assumptions of Proposition~\ref{pr:finite}, described in
\ref{p:setup-finite},
%Remark~\ref{r:maxmod}
are not verified: one has $\Cliff (C) = 2$, as $C$
carries either a $g^1_4$ or a $g^2_6$
\cite[Remarque~4]{beauville-merindol}, see also Remark~\ref{r:unic-BM}. 
\end{remark}

\section{Plane curves with ordinary singularities}
\label{S:plane}

In this Section we construct an extension of plane curves
with $a \leq 9$ ordinary singularities to an
$(11-a)$-dimensional variety, 
and thus give a lower bound on the coranks of their respective Gauss
maps.
The construction proposed in the following proposition is not new,
see, \eg \cite{epema}.

\begin{proposition}
\label{pr:plane-curves}
Let $C \subset \P^2$ be an integral curve with $a \leq 9$
singular points in general position, such that a simple blow-up of
$\P^2$ at
these $a$ points resolves the singularities of $C$
(\eg $C$ has $a$ ordinary singular points in general position and no other
singularities). Assume moreover that $C$ has genus $g \geq 3$. 
There is a family of dimension $9-a$ of mutually non-isomorphic
surfaces in $\P^g$ having the canonical image of the resolution of $C$
as a hyperplane section.
\end{proposition}

\begin{proof}
Let $C \subset \P^2$ be an integral curve of degree $d$ satisfying the
assumptions of the Proposition;
we call
$p_1,\ldots,p_a$ its singular points, and $m_1,\ldots,m_a >1$ the respective
multiplicities of $C$ at these points.
Let $T$ be a
smooth cubic passing through $p_1,\ldots,p_a$.
We call $p_{a+1},\ldots,p_h$ the intersection points,
possibly infinitely near, of $T$ and $C$
off $p_1,\ldots,p_a$, and set $m_{a+1}=\cdots = m_h =1$, so that
$\sum _{i=1}^h m_i = 3d$.
\par
Consider the blow-up $\sigma_T: \tilde \P_T \to \P^2$ at all the
intersection points $p_1,\ldots,p_h$ of $T$ and $C$, and call $E_i$
the exceptional divisor over the point $p_i$
(note that it is a chain of reduced rational curves, such that
$E_i^2=-1$). 
The proper transform $C_T$ of $C$ is smooth by assumption, and
disjoint from the proper transform $\hat T$ of $T$.
The curve $\hat T$ is an anticanonical divisor on $\tilde \P_T$.
Let $H$ be the line class on $\P^2$, and consider the linear system
\begin{equation}
\label{eq:ls-plane}
\textstyle
\bigl| C_T \bigr| =
\bigl| d\cdot \sigma_T^* H - \sum _{i=1}^h m_i E_i \bigr| =
\bigl| (d-3)\sigma_T^* H - \sum _{i=1}^h (m_i-1) E_i + \hat T \bigr|
= \bigl| K_{\tilde \P_T}+C_T + \hat T \bigr |.
\end{equation}
It restricts to the complete canonical series on $C_T$,
and defines a birational map $\phi_T: \tilde \P_T \dashrightarrow
\P^g$.
It follows that the image surface $S_T=\phi_T(\tilde \P_T)$ is an
extension of the canonical model of the resolution of $C$.
The curve $\hat T$ is contracted to an elliptic
singularity by $\phi _T$.

Now the cubic curves passing through $p_1,\ldots,p_a$ form a
linear system of dimension $9-a$, and the generic such cubic
is smooth. The Proposition therefore follows from the
fact that two different choices of $T$ give two non-isomorphic
surfaces $S_T$, which is the content of Lemma~\ref{l:fake-distinct}
below.
\end{proof}

\begin{lemma}
\label{l:fake-distinct}
Maintain the notation of the proof of
Proposition~\ref{pr:plane-curves}, and let $T'$ be another cubic
satisfying the same assumptions as $T$. 
The surfaces $S_T$ and $S_{T'}$ are not isomorphic.
\end{lemma}

\begin{proof}
The cubics through $p_1,\ldots,p_a$ cut a
base-point-free
$g^ {9-a}_{h-a}$ on the normalisation of $C$.
On the other hand, $S_T$ and $S_{T'}$ each have $h-a$ lines,
corresponding respectively to the simple base points of
the linear systems $\sigma_{T*} |C_T|$ and $\sigma_{T'*} |C_{T'}|$ on the
plane. 
\par
If $S_T$ and $S_{T'}$ are isomorphic, the two elliptic
curves $T$ and $T'$ are isomorphic as well.
Also, the isomorphism $S_T \cong S_{T'}$ sends the aforementioned
lines on $S_T$ to their counterparts on $S_{T'}$.
This implies that $T$ and $T'$ cut out the same member of the 
$g^ {9-a}_{h-a}$ on the normalisation of $C$, hence coincide.
\end{proof}

\bigskip
Note that if $h \geq 19$ (recall that $h$ is the number of points
in the set-theoretic intersection $C \cap T$),
the surfaces $S_T$ are neither $K3$ surfaces nor limits of
such, because in this case the curve $\hat T$ is contracted by
$\phi_T$ to an elliptic singularity which is not smoothable;
see \cite{abs1} and the references therein for more details.

\begin{corollary}
\label{c:wahl-conj}
Let $C \subset \P^2$ be an integral curve of geometric genus $g \geq
11$, with $a \leq 9$ singular points in general position, such that a
simple blow-up of $\P^2$ at these $a$ points resolves the
singularities of $C$.
% \eg $C$ has $a$ ordinary singular points in
% general position and no other singularities).
Let $\bar C$ be the normalization of $C$.
% and assume that it has Clifford index $\Cliff (C)>2$. 
One has
\begin{equation}
\label{ineq-plane}
\cork (\Phi_{\bar C}) \geq 10 - a.
\end{equation}
\end{corollary}

\vspace{\negvcorr}
\begin{proof}
We first prove \eqref{ineq-plane} under the assumption that
$\Cliff (C)>2$. 
By Proposition~\ref{pr:plane-curves}, there is a family of dimension
$9-a$ of mutually non-isomorphic extensions of the canonical
model of $\bar C$. By Remark~\ref{r:unicita}, these correspond to
mutually non-isomorphic ribbons $\tilde C$ over $\bar C \subset
\P^{g-1}$, and it follows that
$\P (\ker (\trsp\Phi_C))$ has dimension at least $9-a$, see
\ref{p:unobstr-ribbons}.

On the other hand, if $\Cliff (C) \leq 2$, then $\bar C$ is either
hyperelliptic, trigonal, or tetragonal, since $g \geq 11$. In all
these cases, it is known that $\cork (\Phi_{\bar C}) \geq 10 -a$ by
the results quoted in \ref{p:plane-curves}, except possibly if $a=0$
and $\bar C$ is tetragonal; in the latter case $C$ is necessarily a
smooth plane quintic, in contradiction with the assumption $g\geq
11$. 
\end{proof}

\begin{conjecture}[{\cite[p.~80]{wahl90}}]
\label{c:conjW}
Let $S$ be a regular surface. There should exist an integer $g_0$ such
that for every non-singular curve $C \subset S$ of genus $g \geq g_0$,
one has
\[
\cork (\Phi_C)
\geq \h^0 (S, \omega_S ^{-1}).
\]
\end{conjecture}
\vspace{\mnegvcorr}

\noindent
In the same article, this conjecture is proved for $S=\P^2$
\cite[Thm.~4.8]{wahl90}.

% \par
% Corollary~\ref{c:wahl-conj} solves this conjecture when $S$ is the
% projective plane blown-up at $a \leq 9$ points in general
% position, see Proposition~\ref{pr:wahl-conj} below.
% Actually, the same argument works for 
% any blow-up of the plane having an anticanonical curve $\hat T$ 
% (which is easily seen to have 
% $h^ 0(\mathcal O_{\hat T})=1$), \eg for
% the plane blown-up at any number of points lying on a smooth cubic
% curve $T_0$.
% %and in general position on this cubic.
% We don't dwell on this here.

\begin{proposition}
\label{pr:wahl-conj}
The Wahl Conjecture \ref{c:conjW} holds for any blow-up of the
projective plane having an anticanonical curve, \eg
when $S$ is the projective plane blown-up at $a \leq 9$ points in
general position.
\end{proposition}

% \subparagraph{Proof}\medskip
% \label{pf:wahl-conj}
\begin{proof}
Let $\epsilon: S \to \P^2$ be the blow-up of $a \leq 9$ points in general
position, and set $g_0=11$. One has $\h^0 (S, \omega_S ^{-1})=10-a$.
Let $C$ be a smooth curve in $S$ of genus $g\geq g_0$;
it is the normalization of $\epsilon(C)$, and the latter has at most
$a$ singular points in general position resolved in one single
blow-up, namely $\restr \epsilon C$.
It thus follows from Corollary~\ref{c:wahl-conj} that
$\cork(\Phi_C) \geq 10-a$
as required.
\par The same argument works for 
any blow-up of the plane having an anticanonical curve $\hat T$ 
(which is easily seen to have $h^ 0(\mathcal O_{\hat T})=1$); we leave
this to the reader.
\end{proof}

% is The Conjecture for $\P^2$ blown-up at $a$ general
% points is proven when one knows that the Wahl map $\Phi_{\bar C}$ of
% the normalization $\bar C$ of a plane curve $C$ with $a$ singular
% points resolved in one single blow-up, and no further singularities,
% has corank greater or equal than $10 -a$. 
% \par Corollary~\ref{c:wahl-conj} proves this in case $\bar C$ has
% Clifford index greater than $2$. If not, we may assume that $\bar C$
% is tetragonal by choosing $g_0 = 11$. 
% We may moreover assume $a \geq 1$ by \cite[Thm.~4.8]{wahl90}.
% Then the result follows from \cite{brawner}, in which it is proven
% that if $\bar C$ is tetragonal then $\cork (\Phi_{\bar C}) \geq 9$.
% \endproof

\begin{remark}
Corollary~\ref{c:wahl-conj} 
and Conjecture~\ref{c:conjW}
contradict
\cite[Theorem~B, (ii)]{kang}, which asserts that the Wahl map of the
normalization of a plane curve of degree $d$ with one node and one
ordinary $(d-5)$-fold point, and no other singularity, has corank
$7$.
% (the latter result would, by the way, disprove Wahl's
% Conjecture~\ref{c:conjW} above).
We double-checked, using cohomological methods, that the corank is
indeed greater or equal than $8$ in this case.
% As we were able to obtain a second proof of the fact that the
% corank of the Wahl map in this case is $2$, using cohomological
% methods, we believe that there has been a computation mistake in 
% \cite{kang}.
% %\cite[Theorem~B, (ii)]{kang} is erroneous.
\end{remark}

\begin{example}
For curves $C$ as in Proposition~\ref{pr:plane-curves}, it is possible
to construct a $(9-a)$-extension containing as linear
sections all the surface extensions constructed in the proof of
Proposition~\ref{pr:plane-curves}. This is the universal extension of
$C$ whenever one has equality in \eqref{ineq-plane},
which happens 
when $C$ is smooth \cite[Thm.~4.8]{wahl90}, or has up
to two nodes \cite{kang}, or in various other cases \cite{kang,
  edoardo-DP1}. 
\par Assume for simplicity that $C$ is smooth, and consider the
product $\P^2 \times \P^9$. It contains 
$\mathcal C = C \times \P^9$ and the universal family of plane
cubics $\mathcal T$ over $\P^9 \cong |\O _{\P^2} (3)|$.
We let $\CT$ be the linear system of hypersurfaces of bidegree $(d,1)$
in $\P^2 \times \P^9$ containing the intersection scheme $\mathcal C
\cap \mathcal T$;
we claim that it defines a birational map, the image of which is the 
extension $X \subset \P^{g+9}$ of $C$ we are looking for.
\par
We first observe that the linear system $\CT$
restricts on the fibres of the second projection to the linear systems
\eqref{eq:ls-plane} defining the various extensions of the canonical
model of $C$.
It follows that it defines a birational map, and that its image has as
linear sections the various surfaces images of the linear systems
\eqref{eq:ls-plane}.
Moreover, it maps $\mathcal {T}$ to a $\P^9$; 
for each surface extension $S_T$ of the canonical model of $C$, this
$\P^9$ image of $\mathcal T$ intersects $\vect {S_T} \cong \P^g$ at one
point, which is the elliptic singularity of $S_T$.
\par
On the other hand, the members of $\CT$ restrict to hyperplanes on the
fibres of the first projection; over a point $p \in C \subset \P^2$,
they all restrict to the same hyperplane of $\P^9 \cong |\O _{\P^2}
(3)|$, namely the one parametrizing plane cubics passing through
$p$. It follows that the birational map defined by $\CT$ contracts
$\mathcal C = C \times \P^9$ to $C$.
We leave the remaining details to the reader.
\end{example}

\begin{closing}
% \bibliographystyle{mysmf-alpha}
% \bibliography{wahl}

% debut wahl.bbl
\providecommand{\bysame}{\leavevmode\hbox to3em{\hrulefill}\thinspace}
\providecommand{\og}{``}
\providecommand{\fg}{''}
\providecommand{\smfandname}{and}
\providecommand{\smfedsname}{eds.}
\providecommand{\smfedname}{ed.}
\providecommand{\smfmastersthesisname}{M\'emoire}
\providecommand{\smfphdthesisname}{Th\`ese}

% fin bbl

\medskip\noindent
Ciro Ciliberto.
Dipartimento di Matematica.
Università degli Studi di Roma Tor Vergata.
Via della Ricerca Scientifica,
00133 Roma, Italy.
\texttt{cilibert@mat.uniroma2.it}

\medskip\noindent
Thomas Dedieu.
Institut de Mathématiques de Toulouse~; UMR5219.
Université de Toulouse~; CNRS.
UPS IMT, F-31062 Toulouse Cedex 9, France.
\texttt{thomas.dedieu@math.univ-toulouse.fr}

\medskip\noindent
Edoardo Sernesi.
Dipartimento di Matematica e Fisica.
Università Roma Tre.
Largo S.L.\ Murialdo 1,
00146 Roma, Italy.
\texttt{sernesi@mat.uniroma3.it}

\renewcommand{\thefootnote}{}
\footnotetext
{CC and ThD were membres of project FOSICAV, which 
has received funding from the European Union's Horizon
2020 research and innovation programme under the Marie
Sk{\l}odowska-Curie grant agreement No~652782.}
\end{closing}

\end{document}